\documentclass[10pt]{article}
\usepackage{amssymb}
\usepackage{graphicx}
\usepackage{xcolor} % A package to add color.
\usepackage{tensor}
\usepackage{fullpage} % Sets all margins to 1 in.  
\usepackage{amsmath}
\usepackage{amsthm}
\usepackage{verbatim}
\usepackage{txfonts,bm,mathtools}
\providecommand{\bysame}{\leavevmode\hbox to3em{\hrulefill}\thinspace}
\providecommand{\MR}{\relax\ifhmode\unskip\space\fi MR }
\providecommand{\href}[2]{#2}
\DeclareFontFamily{U}{mathx}{\hyphenchar\font45}
\DeclareFontShape{U}{mathx}{m}{n}{
	<5> <6> <7> <8> <9> <10>
	<10.95> <12> <14.4> <17.28> <20.74> <24.88>
	mathx10
}{}
\DeclareSymbolFont{mathx}{U}{mathx}{m}{n}
\DeclareFontSubstitution{U}{mathx}{m}{n}
\DeclareMathAccent{\widecheck}{0}{mathx}{"71}
\DeclareMathAccent{\wideparen}{0}{mathx}{"75}

\usepackage{enumitem}
\setlist[enumerate]{leftmargin=1.5em}
\setlist[itemize]{leftmargin=1.5em}

\definecolor{green}{rgb}{0,0.8,0} % Redefines the color green.
\newtheorem{theorem}{Theorem}[section]

\newtheorem{lemma}[theorem]{Lemma}
\newtheorem{proposition}[theorem]{Proposition}

\theoremstyle{definition}

\theoremstyle{remark}
\newtheorem{remark}[theorem]{Remark}

\numberwithin{equation}{section}
\makeatletter
\newcommand{\nrm}{\@ifstar{\nrmb}{\nrmi}}
\newcommand{\nrmi}[1]{\Vert{#1}\Vert}
\newcommand{\nrmb}[1]{\left\Vert{#1}\right\Vert}
\newcommand{\abs}{\@ifstar{\absb}{\absi}}
\newcommand{\absi}[1]{\vert{#1}\vert}
\newcommand{\absb}[1]{\left\vert{#1}\right\vert}
\newcommand{\brk}{\@ifstar{\brkb}{\brki}}
\newcommand{\brki}[1]{\langle{#1}\rangle}
\newcommand{\brkb}[1]{\left\langle{#1}\right\rangle}
\newcommand{\set}{\@ifstar{\setb}{\seti}}
\newcommand{\seti}[1]{\{#1\}}
\newcommand{\setb}[1]{\left\{ #1\right\}}
\makeatother

\newcommand{\nnrm}[1]{{\vert\kern-0.25ex\vert\kern-0.25ex\vert #1 
    \vert\kern-0.25ex\vert\kern-0.25ex\vert}}

\newcommand{\VERT}[1]{{\left\vert\kern-0.25ex\left\vert\kern-0.25ex\left\vert #1 
    \right\vert\kern-0.25ex\right\vert\kern-0.25ex\right\vert}}

\newcommand{\lap}{\Delta}

\newcommand{\ud}{\mathrm{d}}
\newcommand{\rd}{\partial}
\newcommand{\nb}{\nabla}

\def\ct{\cos(\theta)}

\def\bR{\bar{r}}

\def\bz{\bar{z}}

\def\bbx{\bold{x}}
\def\bby{\bold{y}}

\newcommand{\Gmm}{\Gamma}

\newcommand{\varep}{\varepsilon}

\newcommand{\tht}{\theta}

\newcommand{\bbR}{\mathbb R}

\newcommand{\calF}{\mathcal F}

\newcommand{\dd}{\textrm{d}}

\begin{document}

\title{Axi-symmetric solutions for active vector models \\ generalizing 3D Euler
	and electron--MHD equations}%: Title of the article
\author{Dongho Chae\thanks{Department of Mathematics, Chung-ang University. E-mail: dchae@cau.ac.kr} \and Kyudong Choi\thanks{Department of Mathematical Sciences, Ulsan National Institute of Science and Technology, kchoi@unist.ac.kr} 
	\and  In-Jee Jeong\thanks{Department of Mathematical Sciences and RIM, Seoul National University.  E-mail: injee\_j@snu.ac.kr}}

\date{December 23, 2022} 
% ----------------------------------------------------------------
\maketitle
% ----------------------------------------------------------------

\begin{abstract}
	We study systems interpolating between the 3D incompressible Euler and electron--MHD equations, given by \begin{equation*}
		\left\{
		\begin{aligned}
			& \rd_t B + V \cdot \nb B = B\cdot \nb V,  \\
			& V = -\nb\times (-\lap)^{-a} B, \\
			& \nabla\cdot B = 0,
		\end{aligned}
		\right.
	\end{equation*} where $B$ is a time-dependent vector field in $\bbR^3$. Under the assumption that the initial data is axi-symmetric without swirl, we prove local well-posedness of Lipschitz continuous solutions and existence of traveling waves in the range $1/2<a<1$. These generalize the corresponding results for the 3D axisymmetric Euler equations and should be useful in the study of stability and instability for axisymmetric solutions. 
\end{abstract}

\section{Introduction}
\subsection{Active vector system}
In this paper, we investigate the following {active vector systems} introduced in \cite{CJ1}: \begin{equation}\label{eq:EH-active}
	\left\{
	\begin{aligned}
		& \rd_t B + V \cdot \nb B = B\cdot \nb V,  \\
		& V = -\nb\times (-\lap)^{-a} B, \\
		& \nabla\cdot B = 0, 
	\end{aligned}
	\right.
\end{equation}where $B(t,\cdot):\bbR^3\rightarrow\bbR^3$ is a time-dependent divergence free vector field defined in $\bbR^3$. Alternatively, the system can be written in the more compact form \begin{equation}\label{eq:EH}
	\left\{
	\begin{aligned}
		& \rd_t B + \nb\times ( ( \nb\times (-\lap)^{-a} B ) \times B) = 0,  \\
		& \nabla \cdot B = 0. \\ 
	\end{aligned}
	\right.
\end{equation} The system \eqref{eq:EH-active} generalizes three important physical models, namely the three-dimensional incompressible Euler equations, the electron magneto-hydrodynamics (E--MHD) system, and the entire family of generalized surface quasi-geostrophic (SQG) equations; see \S \ref{subsec:motiv} below for details. The recent paper \cite{CJ1} gave local well-posedness of smooth solutions for \eqref{eq:EH-active} in the range $\frac12 \le a \le 1$, and more generally when $(-\lap)^{-a} B$ is replaced with $\Gmm[B]$ where $\Gmm$ is a Fourier multiplier with a radial symbol which lies between $(-\lap)^{-1}$ and $(-\lap)^{-\frac{1}{2}}$. In the current work, we focus on the case $(-\lap)^{-a}$ with $\frac{1}{2}< a <1$ and initiate the study of \textit{axi-symmetric} solutions, towards a better understanding of the issue of global regularity versus finite-time singularity formation. The study of axisymmetric Euler equation (corresponding to the case $a=1$) is a very classical subject, and the axisymmetric system in the E--MHD case ($a=0$) was recently studied in \cite{ChaeWeng1,ChaeWeng2,JO1,LiPa}. We shall provide several related works and motivations for the study of axi-symmetric solutions in \S \ref{subsec:motiv} after stating our main results in \S \ref{subsec:main}. 
%Note that since we are assuming that the symbol of $\Gmm$ is radial, (formally) rotational symmetries of the initial data are preserved in time.

\subsection{Axi-symmetric systems}\label{subsec:axi}

By axi-symmetry, we mean solutions which are left invariant under all rotations fixing an axis; employing the cylindrical coordinate system $(r,\tht,z)$, we say that a vector field $B$ is axi-symmetric if it has the form \begin{equation*}\label{eq:B-axi}
	\begin{split}
		B =  b^{r}(r,z)e_{r} + b^{\tht} (r,z) e_{\tht} + b^{z}(r,z) e_{z}. 
	\end{split}
\end{equation*} Then, we similarly have \begin{equation*}\label{eq:V-axisym}
\begin{split}
	V = v^{r}(r,z)e_{r} + v^{\tht} (r,z) e_{\tht} + v^{z}(r,z) e_{z}.
\end{split}
\end{equation*} It will be useful to introduce the variable $U$ defined by the unique solution of \begin{equation}\label{eq:U-def}
\left\{
\begin{aligned}
	&\nb \times U  = B, \\
	&\nb \cdot U = 0
\end{aligned}
\right.
\end{equation} decaying at infinity. This is precisely the fluid velocity in the case of 3D Euler, namely when $a=1$. Under the assumption of axi-symmetry, we write \begin{equation*}\label{eq:U-axi}
\begin{split}
	U =  u^{r}(r,z)e_{r} + u^{\tht} (r,z) e_{\tht} + u^{z}(r,z) e_{z}. 
\end{split}
\end{equation*}  
Then, it turns out that \eqref{eq:EH-active} can be written as a system of equations for $b^\tht$ and $u^\tht$:
\begin{equation}\label{eq:EH-axisym}
	\left\{
	\begin{aligned}
	&\rd_t b^\tht + (v^r\rd_r + v^z \rd_z) b^\tht = \frac{v^r}{r} b^\tht - \frac{b^{r}}{r} v^{\tht} + (b^r\rd_r + b^z\rd_z) v^\tht ,\\
	&\rd_t u^\tht + (v^r\rd_r + v^z \rd_z) u^\tht = - \frac{v^r}{r} u^\tht. 
	\end{aligned}
	\right.
\end{equation} The system is closed in terms of $(b^\tht,u^\tht)$ since $u^\tht$ at each moment of time determines $b^r$ and $b^z$ via $b^r = -\rd_z u^\tht$ and $b^z = r^{-1}\rd_{r}(ru^\tht),$ which follows from \eqref{eq:U-def}. In turn, components of $B$ determine those of $V$, see \eqref{eq:VB} below.  We have that $u^\tht \equiv 0$ if and only if $v^\tht \equiv 0$, and therefore one may observe that the ansatz $u^\tht \equiv 0$ propagates in time for \eqref{eq:EH-axisym}. Under this assumption, we obtain the simpler \textit{no--swirl} equation: \begin{equation}\label{eq:EH-axisym-noswirl}
\begin{split}
	\rd_t b^\tht + (v^r\rd_r + v^z \rd_z) b^\tht = \frac{v^r}{r} b^\tht . 
\end{split}
\end{equation} For convenience of the reader, we provide a derivation of \eqref{eq:EH-axisym} below. For 3D axi-symmetric Euler equations, such a formulation is well-known; see for instance \cite[Chap. 5]{MB}.

\subsection{Main results}\label{subsec:main}

The main results of this paper are local well-posedness in the Lipschitz function class and existence of traveling wave solutions in the same function class, in the range $\frac12<a\le1$. The latter result generalizes known solutions in the three-dimensional Euler equations $(a=1)$ and in particular shows that there are at least some non-trivial global strong solutions to the generalized systems \eqref{eq:EH-active}. 

To state our results, we define a space of Lipschitz continuous functions:  \begin{equation*}\label{def:Lip}
	\begin{split}
		X^1(\bbR^3) = \left\{ f \in L^1 \cap L^\infty(\bbR^3)\, :\, \nb f \in L^\infty(\bbR^3) \right\}. 
	\end{split}
\end{equation*}

For axisymmetric no-swirl equation, we provide local existence and uniqueness of solutions in the class $X^1$. 
\begin{theorem}[Well-posedness in the Lipschitz class] \label{thm:lwp}
	Let  {$\frac 1 2 < a <1$.} Assume that $B_0 \in X^1$ is a divergence-free vector field in $\bbR^3$. Then, for any $T>0$, there is at most one solution to \eqref{eq:EH-active} belonging in $L^\infty([0,T];X^1)$ corresponding to $B_0$. Furthermore, if we assume in addition that $B_0$ takes the form  $b^\tht_{0}(r,z) e_\tht$, there exist $T = T(\nrm{B_0}_{X^1})>0$ and a solution $B = b^\tht e_\tht$ to \eqref{eq:EH-active} in   $L^\infty([0,T];X^1)$, where $b^\tht$ solves \eqref{eq:EH-axisym-noswirl}. 
\end{theorem}

We emphasize that while the \textit{existence} could be shown only for axisymmetric and no-swirl initial data, the \textit{uniqueness} results holds not only within the axisymmetric equation but for the full system \eqref{eq:EH-active} (or equivalently \eqref{eq:EH}). This in particular rules out the possibility of non-axisymmetric Lipschitz solutions corresponding to an axisymmetric Lipschitz data. 

Our second main result proves existence of non-trivial traveling wave solutions which belong to the above well-posedness class. For this, it will be convenient to introduce   the \textit{relative vorticity} $\xi$: if $B(t,\bbx) =  b^{\tht} (t,r,z) e_{\tht}  $, we define $\xi(t,r,z)=b^\tht(t,r,z)/r.$ Then, by interpreting $\xi(t,\cdot)$ as an axi-symmetric function on $\mathbb{R}^3$, \eqref{eq:EH-axisym-noswirl} simply  becomes 
\begin{equation}\label{eq_xi_evol}
	\partial_t \xi + V\cdot \nabla_{\mathbf{x}} \xi=0,\quad t>0,\quad \bbx\in\mathbb{R}^3
\end{equation}
for the divergence free velocity $V = - (-\lap)^{-a} \nb\times B$. 
\begin{theorem}[Lipschitz continuous traveling waves]\label{thm_TW_exist}
	Let  {$\frac 1 2 < a <1$.} Then, there exists a non-trivial traveling wave solution $\xi(t,\bbx)=\bar\xi(x-Wte_z)$ for some $W>0$ of \eqref{eq_xi_evol} such that
	\begin{equation*}\begin{split}
			\bar\xi\in L^\infty(\mathbb{R}^3), \mbox{ Lipschitz continuous, non-negative, and compactly supported.}\\
			%&\mbox{spt}(\bar\xi)\quad\mbox{is compact in } \mathbb{R}^3.
	\end{split}\end{equation*}
\end{theorem}
The proof is based  on a variational approach. In particular, we shall follow the classical variational setup of Friedman--Turkington \cite[Theorem 2.2]{FT81}, see \S \ref{subsec_variation}.
%\begin{remark} 	For the case $a=1$ (3D axi-symmetric Euler), there have been many existence results of traveling waves including 	\cite{Hill, Norbury72, MR583638, FT81,  Burton_03, Burton_Preciso}. In order to obtain Lipschitz solution for $\frac 1 2 <a<1$, we adopt the setting of Friedman--Turkington \cite[Theorem 2.2]{FT81}. See Subsection \ref{subsec_variation}.\end{remark}

\subsection{Discussion}\label{subsec:motiv}

Let us discuss several motivations to investigate traveling wave solutions of the system \eqref{eq:EH-active}. 

\begin{enumerate}
	\item \textit{Existence and stability of traveling waves}. For the axi-symmetric Euler equations, there is a huge body of literature on the existence and stability of traveling wave solutions, which are often referred to as \textit{vortex rings} (see \cite{Hill, Norbury72, Ni80, FT81,  Burton_03, Burton_Preciso, CWWZ, Choi_hill} and references therein). They correspond to ring-shaped vortex structures which can be observed in nature. Furthermore, such traveling waves can exist with non-zero \textit{swirl} component of the velocity (\cite{swirl1,swirl2,swirl3}), which is interesting since they provide non-trivial examples of global strong solutions to the 3D Euler equations. The proof of existence and stability is usually achieved by a variational argument: after defining an appropriate admissible class with an energy function, the basic strategy is to solve the maximization problem. When a maximum exists, we have a traveling wave, and when the maximum exists uniquely, we obtain its stability. For the two-dimensional case, traveling waves have been studied for instance in \cite{BNL13, Burton2021, CLZ2021,AC_lamb} for the Euler equations and \cite{CQZZ_imrn, CQZZ_jfa} for the generalized SQG equations. 
	
	Since the system \eqref{eq:EH-active} is a simultaneous generalization of 3D Euler and 2D SQG equations, our result on the existence of traveling waves can be considered as a simultaneous extension of existing works as well. \\
	
	%For stability for explicitly written vortices, there have been recent progresses \cite{AC_lamb} for Chaplygin-Lamb dipoles and  \cite{Choi_hill} for Hill's spherical vortices. 
	%Uniformly rotating solutions in 2D Euler equations also can have orbital stability (e.g. Kirchhoff's ellipses \cite {Tang}, Kelvin's $m$-waves \cite{Wan86,CJ_Kelvin}).
	
	\item \textit{Well-posedness issues}. The system \eqref{eq:EH-active} gives rise to challenging local and global well-posedness issues. The pioneering work \cite{ChaeWeng1} obtained finite-time singularity formation for smooth data in the case $a = 0$, assuming well-posedness for axisymmetric data. Remarkably, the axisymmetric system in this case is simply given by the inviscid Burgers equation. While \cite{JO1} proves that the initial value problem for \eqref{eq:EH-active} in the case $a = 0$ is \textit{strongly ill-posed} in Sobolev (and even in any Gevrey) spaces, it does not exclude the possibility that there is a very special type of local well-posedness near the initial data used in \cite{ChaeWeng1}. The ill-posedness is expected to hold as well in the range $a<1/2$, but this seems to be a very technically challenging problem. On the other hand, \cite{CJ1} obtained local well-posedness in sufficiently regular Sobolev spaces when $a\ge1/2$. The uniqueness criterion is improved to the Lipschitz class in the current work, assuming $a>1/2$. 
	\\
	
	\item \textit{Infinite norm growth via instability of traveling waves}. In the past decade, there have been significant progress on the problem of \textit{infinite norm growth}, or in other words, creation of arbitrarily small scales for fluid models (\cite{Den,Den2,Den3,Denisov-merging,KN,KS,Z,Xu,KRYZ,EJ1,JY,JY2,KY,CJ_Kelvin,CJ_perimeter,CJ_Hill,CJ-axi,CJL_half,Tam-axi,EJSVP2,EJE,EJO,Elgindi-3D,ChenHou}). We refer the interested reader to the excellent survey papers \cite{Ki-sur,Ki-sur2,DE}. A common feature of these works is that the instability result is based on a stability statement of the ``background'' solution in a weaker topology. In unbounded domains like $\bbR^2$ and $\bbR^3$, traveling waves constructed by variational principles are perfect candidates for such background flows; they are expected to be stable in some $L^p$ norms of the vorticity, which gives a pointwise control on the velocity perturbation; see \cite{CJ_lamb,CJ_Hill}.
	
	%We can deduce Lagrangian instability from stability results of traveling waves. Indeed, we use $L^p$-stability  at the vortex level, which ensures that  difference in velocity perturbation is uniformly controlled. This uniform small difference in velocity level produces the large difference at   trajectory  level as time goes.
	
	%For instance, there are vortices with infinite perimeter growth (for infinite time) near Chaplygin-Lamb dipoles  \cite{CJ_lamb} and  Hill's spherical vortices  \cite{CJ_Hill}.
	% (also see \cite{CJL_half} for half strip). 
	 %For 2D rotating solutions, we have  finite growth  (see \cite{CJ_perimeter} for disks and  \cite{CJ_Kelvin} for Kelvin's $m$-waves).  
	 
	 Therefore, it is expected that the traveling waves constructed in this work can be applied to produce infinite growth results for smooth solutions to \eqref{eq:EH-active}. 
\end{enumerate}

\subsection{Derivation of axi-symmetric system} We briefly discuss how to arrive at the system \eqref{eq:EH-axisym}. From the following formula for the convective derivative
\begin{equation*}\label{eq:conv-der-axisym}
	\begin{split}
		(V\cdot\nb)B & = \left( v^r \rd_rb^r + v^z\rd_z b^r - r^{-1} v^{\tht}b^{\tht} \right)e_{r}  + \left( v^r \rd_rb^{\tht} + v^z\rd_z b^{\tht} + r^{-1}v^{\tht}b^{r} \right)e_{\tht} + \left(v^r \rd_rb^{z} + v^z\rd_z b^{z}  \right)e_{z} ,
	\end{split}
\end{equation*} we obtain the equation of $b^\tht$ in \eqref{eq:EH-axisym} by taking the $\tht$--component of the equation of $B$ in \eqref{eq:EH-active}. Next, by ``integrating'' the form \eqref{eq:EH}, we obtain the following equation of $U$
\begin{equation*}  
	\begin{split}
		\rd_t U + (\nb \times U) \times V + \nb q = 0,
	\end{split}
\end{equation*} with some scalar-valued function $q$. Then, taking the $\tht$--component of this equation gives the equation of $u^\tht$ in \eqref{eq:EH-axisym}. Next, let us demonstrate the relationship between $V$, $U$, and $B$, which allows us to close the system \eqref{eq:EH-axisym} in terms of $(b^\tht,u^\tht)$. 

\medskip

\noindent \textbf{Relation between $V$ and $U$.} From the relation between ``two velocities"
$V = \Gmm\lap U = - (-\lap)^{1-a} U$, we have 
\begin{equation*}\label{eq:VU}
	\begin{split}
		V(x) = C_{a}\int_{\bbR^{3}} \frac{U(x)-U(y)}{|x-y|^{5-2a}} \, \ud y .
	\end{split}
\end{equation*} Taking the second component, evaluating at $x = (r,0,z)$ and switching to the cylindrical coordinates, \begin{equation*}\label{eq:vtht-utht}
	\begin{split}
		v^{\tht}(r,z) = C_{a} \int_{\bbR} \int_{\bbR^+} \int_{0}^{2\pi}  \frac{ u^{\tht}(r,z) -  u^{\tht}(r',z') \cos(\tht')}{(r^2 - 2rr'\cos(\tht')+ (r')^2 + (z-z')^2 )^{\frac{5}{2}-a}}  r' \, \ud\tht'  \, \ud r' \, \ud z'.
	\end{split}
\end{equation*} Similarly, taking $x = (r,0,z)$ and the first and third components give \begin{equation*}\label{eq:vr-ur}
	\begin{split}
		v^{r}(r,z) = C_{a} \int_{\bbR} \int_{\bbR^+} \int_{0}^{2\pi}  \frac{ u^{r}(r,z) - u^{r}(r',z') \cos(\tht') }{(r^2 - 2rr'\cos(\tht')+ (r')^2 + (z-z')^2 )^{\frac{5}{2}-a}}  r' \, \ud\tht'  \, \ud r' \, \ud z'
	\end{split}
\end{equation*} and \begin{equation*}\label{eq:vz-uz}
	\begin{split}
		v^{z}(r,z) = C_{a} \int_{\bbR} \int_{\bbR^+} \int_{0}^{2\pi}  \frac{ u^{z}(r,z) - u^{z}(r',z') }{(r^2 - 2rr'\cos(\tht')+ (r')^2 + (z-z')^2 )^{\frac{5}{2}-a}}   r' \, \ud\tht'  \, \ud r' \, \ud z',
	\end{split}
\end{equation*} respectively. 

\medskip

\noindent \textbf{Relation between $V$ and $B$.} 
From $V = - (-\lap)^{-a} \nb\times B$, we have \begin{equation}\label{eq:VB}
	\begin{split}
		V(x) = - \frac{C_a}{2(1-a)}\int \frac{x-y}{|x-y|^{-2a+5}}  \times B(y) \, \ud y. 
	\end{split}
\end{equation} In the case $a = \frac12$, this integral must be defined in the sense of principal value. Taking the first component of \eqref{eq:VB} and evaluating at $x = (r,0,z)$, we obtain the following useful relation between $b^\tht$ and $v^r$: \begin{equation*}\label{eq:vr-btht}
	\begin{split}
		v^r(r,z) = \frac{C_a}{2(1-a)}  \int_{\bbR} \int_{\bbR^+} \int_{0}^{2\pi}  \frac{ (z-z')\cos(\tht') b^\tht(r',z') }{(r^2 - 2rr'\cos(\tht')+ (r')^2 + (z-z')^2 )^{\frac{5}{2}-a}}   r' \, \ud\tht'  \, \ud r' \, \ud z' . 
	\end{split}
\end{equation*} 

\subsection*{Organization of the paper}
We shall provide proofs of Theorems \ref{thm:lwp} and \ref{thm_TW_exist} in \S \ref{sec:lwp} and \ref{sec:travel}, respectively.

\section{Local well-posedness}\label{sec:lwp}

\begin{proof}[Proof of Theorem \ref{thm:lwp}]
	
	We divide the proof into a few steps.

	\medskip
	
	\noindent \textbf{1. A priori estimates}. We assume that there is a solution $b^\tht$ to \eqref{eq:EH-axisym-noswirl} satisfying $B = b^\tht e_\tht \in L^\infty([0,T];X^1(\bbR^3))$. To begin with, since $\tilde{v}$ is divergence free in $\bbR^3$, we obtain \begin{equation*}
		\begin{split}
			\frac{d}{dt} \nrm{b^\tht}_{L^1} \le \nrm{  \frac{v^r}{r} b^\tht }_{L^1} \le \nrm{  \frac{v^r}{r} }_{L^\infty} \nrm{ b^\tht}_{L^1}. 
		\end{split}
	\end{equation*} We estimate \begin{equation*}
	\begin{split}
		\nrm{  \frac{v^r}{r} }_{L^\infty} \le C \nrm{\nb V}_{L^\infty} \le C \nrm{B}_{L^\infty}^{1-2a} \nrm{\nb B}_{L^\infty}^{2a} \le C\nrm{B}_{X^1}. 
	\end{split}
\end{equation*} This gives \begin{equation*}
\begin{split}
		\frac{d}{dt} \nrm{b^\tht}_{L^1} \le C\nrm{B}_{X^1}\nrm{ b^\tht}_{L^1}. 
\end{split}
\end{equation*} Similarly, we may obtain that \begin{equation*}
\begin{split}
			\frac{d}{dt} \nrm{b^\tht}_{L^\infty} \le C\nrm{B}_{X^1}\nrm{ b^\tht}_{L^\infty} 
\end{split}
\end{equation*} holds. Next, we write down the equations for $r^{-1} b^\tht$, $\rd_r b^\tht$, and $\rd_z b^\tht$: 
	\begin{equation}\label{eq:rd-tht-b}
		\begin{split}
				\rd_t \frac{b^\tht}{r} + (v^r\rd_r + v^z \rd_z) \frac{ b^\tht}{r} = 0, 
		\end{split}
	\end{equation} \begin{equation}\label{eq:rd-r-b}
	\begin{split}
			\rd_t \rd_r b^\tht +(v^r\rd_r + v^z \rd_z)\rd_r b^\tht =  -\frac{v^r}{r} \frac{b^\tht}{r} + \rd_r v^r \frac{b^\tht}{r} + \frac{v^r}{r} \rd_r b^\tht - \rd_r\widetilde{v} \cdot \widetilde{\nb} b^\tht,
	\end{split}
\end{equation} and \begin{equation*}\label{eq:rd-z-b}
\begin{split}
	\rd_t \rd_z b^\tht + (v^r\rd_r + v^z \rd_z) \rd_z b^\tht =  \rd_z v^r\frac{b^\tht}{r}  + \frac{v^r}{r} \rd_z b^\tht - \rd_z\widetilde{v} \cdot \widetilde{\nb} b^\tht  . 
\end{split}
\end{equation*} Using the fact that $\nrm{\nb B}_{L^\infty} $ is equivalent, up to constants, with \begin{equation*}
\begin{split}
	\nrm{ r^{-1} b^\tht}_{L^\infty} + \nrm{ \rd_zb^\tht}_{L^\infty} +  \nrm{\rd_rb^\tht}_{L^\infty} 
\end{split}
\end{equation*} and estimating the right hand sides of \eqref{eq:rd-tht-b}--\eqref{eq:rd-r-b} gives \begin{equation*}
\begin{split}
	\frac{d}{dt} \left( 	\nrm{\nb B}_{L^\infty} \right) \le C \nrm{\nb V}_{L^\infty} \nrm{\nb B}_{L^\infty} \le C 	\nrm{ B}_{X^1}	\nrm{\nb B}_{L^\infty}.
\end{split}
\end{equation*}  Combining the estimates, we arrive at \begin{equation}\label{eq:apriori-X1}
\begin{split}
	\frac{d}{dt} 	\nrm{ B}_{X^1}\le C	\nrm{ B}_{X^1}^{2}. 
\end{split}
\end{equation} In particular, there is $T>0$ depending only on $\nrm{B_0}_{X^1}$ such that $\nrm{B}_{L^\infty([0,T];X^1)} \le 2 \nrm{B_0}_{X^1}$.

\medskip

\noindent \textbf{2. Existence}. To prove the existence of a solution in $L^\infty([0,T];X^1)$, we consider the sequence of mollified initial data $\left\{ B_0^{\varepsilon} \right\}$. (Here, we can take the convolution with a radial mollifier, which guarantees that the mollified sequence of data is axisymmetric without swirl.) Since each $B_0^\varep$ is smooth and decaying at infinity, there exists a unique corresponding local smooth solution to \eqref{eq:EH-axisym-noswirl} (\cite{CJ1}), which we denote by $B^\varep$. For each $\varep>0$, the a priori estimate \eqref{eq:apriori-X1} can be justified, which gives a uniform time  $T>0$ depending only on $\nrm{B_0}_{X^{1}}$ such that the solution $B^\varep$ remains smooth on $[0,T]$ and satisfies the uniform bound \begin{equation}\label{eq:uniform-X1}
	\begin{split}
		\nrm{B^\varep}_{L^\infty([0,T];X^1)} \le 2\nrm{B_0}_{X^1}. 
	\end{split}
\end{equation} Here, we used that the Lipschitz norm of $B^\varep$ is a blow-up criterion for \eqref{eq:EH-axisym-noswirl}. We can then extract a subsequence, still denoted by $\{ B^\varep \}$, such that $B^\varep$ is locally uniformly convergent in $C([0,T]\times \bbR^3)$. Denoting the limit by $B$, we first see that $B(t=0)=B_0$. Next, from the uniform bound \eqref{eq:uniform-X1} and pointwise convergence, we obtain that \begin{equation*}
\begin{split}
	\nrm{B}_{L^\infty([0,T];X^1)} \le 2\nrm{B_0}_{X^1}. 
\end{split}
\end{equation*} Lastly, using a weak formulation of \eqref{eq:EH-axisym-noswirl}, it is not difficult to see that $B$ is a weak solution to \eqref{eq:EH-axisym-noswirl}. This finishes the proof of existence. 

\medskip

\noindent \textbf{3. Uniqueness}. To prove uniqueness, we assume that there are two solutions $B_1$ and $B_2$ to \eqref{eq:EH-active} which belong to $L^\infty([0,T];X^1)$ for some $T>0$, corresponding to the same initial data $B_0 \in X^1(\bbR^3)$. We need to used the velocity formulation: denote the corresponding velocity equations by \begin{equation*}
	\begin{split}
				\rd_t U_i + (\nb \times U_i) \times V_i + \nb q_i = 0,
	\end{split}
\end{equation*} for $i=1,2$. We write the equation for the difference $D = U_1 - U_2$: \begin{equation*}\label{eq:diff}
\begin{split}
	\rd_t D + (\nb \times D) \times V_1 + (\nb \times U_2) \times (V_1-V_2) + \nb (q_1-q_2)=0.  
\end{split}
\end{equation*} Taking the dot product of both sides with $D$ and integrating, \begin{equation*}
\begin{split}
		\frac12 \frac{d}{dt} \nrm{D}_{L^2}^2 = I + II + III, 
\end{split}
\end{equation*} with \begin{equation*}
\begin{split}
	I = - \int ((\nb \times D) \times V_1 ) \cdot D,
\end{split}
\end{equation*}\begin{equation*}
\begin{split}
	II = - \int (  (\nb \times U_2) \times (V_1-V_2) ) \cdot D, 
\end{split}
\end{equation*}and \begin{equation*}
\begin{split}
	III = -\int \nb (q_1-q_2) \cdot D = 0. 
\end{split}
\end{equation*} Note that \begin{equation*}
\begin{split}
	I = \int (D \times (\nb \times D)) \cdot V_1, 
\end{split}
\end{equation*} and the pointwise identity \begin{equation*}
\begin{split}
	D \times (\nb \times D) =\sum_{i\ne j}  \left( \frac12\left( \rd_i(D_j^2) \right) - \rd_i(D_iD_j) + D_j \rd_iD_i \right) . 
\end{split}
\end{equation*} The last term can be written as \begin{equation*}
\begin{split}
	\sum_{i\ne j} D_j(\rd_i D_i) = \sum_{i,j} D_j(\rd_i D_i) - \sum_j \frac12 \rd_i (D_i)^2 =  - \sum_j \frac12 \rd_i (D_i)^2
\end{split}
\end{equation*} using $\sum_i \rd_iD_i=0$. Therefore, we can integrate by parts to obtain \begin{equation*}
\begin{split}
		\left|I\right| \le C\nrm{\nb V_1}_{L^\infty} \nrm{D}_{L^2}^2. 
\end{split}
\end{equation*} Next, we rewrite $II$ as \begin{equation*}
\begin{split}
	II = \int ( B_2 \times (-\lap)^{1-a}D) \cdot D
\end{split}
\end{equation*} by recalling $B_2= \nb \times U_2$. This can be further written as \begin{equation*}
\begin{split}
	II = \frac12 \int [ B_{2}\times, (-\lap)^{1-a}] D \cdot D,
\end{split}
\end{equation*} and we see that \begin{equation*}
\begin{split}
	[ B_{2}\times, (-\lap)^{1-a}] D \, (x) = C_a \int \frac{(B_{2}(x)-B_{2}(y))}{|x-y|^{3+2(1-a)}} \times D(y) dy. 
\end{split}
\end{equation*} To estimate this function, we introduce a cutoff $0 \le \chi(|x-y|) \le 1$ which is equal to 1 for $|x-y|<R$, 0 for $|x-y|>2R$ for some $R>0$ to be determined. Then, we estimate \begin{equation*}
\begin{split}
	\left| \int \chi(|x-y|)  \frac{(B_{2}(x)-B_{2}(y))}{|x-y|^{3+2(1-a)}} \times D(y) dy\right| \le C\nrm{\nb B_{2}}_{L^\infty} \int \frac{ \chi(|x-y|)}{|x-y|^{3+2(1-a)-1}} |D(y)| dy,
\end{split}
\end{equation*} using the mean value theorem for $B_{2}$. Then, the kernel is integrable, and we obtain that \begin{equation*}
\begin{split}
	\left\Vert \int \frac{ \chi(|x-y|)}{|x-y|^{3+2(1-a)-1}} |D(y)| dy \right\Vert_{L^2_x} \le C R^{2a-1}  \nrm{D}_{L^2}
\end{split}
\end{equation*} using Young's inequality. Next, we estimate \begin{equation*}
\begin{split}
	\left| \int \frac{(B_{2}(x)-B_{2}(y))}{|x-y|^{3+2(1-a)}} \left( 1 - \chi(|x-y|) \right) \times D(y) dy \right| \le C\nrm{B_{2}}_{L^\infty} \int \frac{ 1- \chi(|x-y|)}{|x-y|^{3+2(1-a)}} |D(y)| dy
\end{split}
\end{equation*}and  \begin{equation*}
\begin{split}
	\left\Vert \int \frac{ 1- \chi(|x-y|)}{|x-y|^{3+2(1-a)}} |D(y)| dy \right\Vert_{L^2_x} \le CR^{2a-2} \nrm{D}_{L^2}. 
\end{split}
\end{equation*} Combining the estimates and choosing $R = \nrm{B_{2}}_{L^\infty}/\nrm{\nb B_{2}}_{L^\infty} $, we conclude that \begin{equation*}
\begin{split}
	\nrm{ [ B_{2}\times, (-\lap)^{1-a}] D}_{L^2} \le C_a \nrm{\nb B_{2}}_{L^\infty}^{2-2a} \nrm{B_{2}}_{L^\infty}^{2a-1} \nrm{D}_{L^2}. 
\end{split}
\end{equation*} Therefore, using the bounds for $I$ and $II$, \begin{equation*}
\begin{split}
	\left| 	\frac12 \frac{d}{dt} \nrm{D}_{L^2}^2  \right| \le C_a\left(  \nrm{\nb V_1}_{L^\infty} + \nrm{\nb B_{2}}_{L^\infty}^{2-2a} \nrm{B_{2}}_{L^\infty}^{2a-1}    \right) \nrm{D}_{L^2}^2 \le C_a (\nrm{B_1}_{X^1} + \nrm{B_2}_{X^1} ) \nrm{D}_{L^2}^2. 
\end{split}
\end{equation*} Using Gronwall's inequality with $D(t=0)=0$ gives $\nrm{D}_{L^2}(t)=0$, which concludes the proof of uniqueness. \end{proof}

\section{Existence of traveling waves}\label{sec:travel}

In this section, we prove existence of Lipschitz  traveling wave solutions of \eqref{eq_xi_evol}. In the sequel, $C_a$ is some positive constant depending only on the value of $a$.
\subsection{Stream function and energy}

For 
$B(\bbx) =  b^{\tht} (r,z) e_{\tht}  $, the corresponding vector stream $\Phi := (-\lap)^{-a}   B
$ has the form
$$
\Phi(\bbx)=C_a\int_{\mathbb{R}^3}\frac{1}{|\bbx-\bby|^{3-2a}} B(\bby) \ud \bby=\phi^\tht(r,z)e_\tht(\tht).
$$ We define the (scalar) stream $\psi$ and the (relative) vorticity $\xi$ by $$\psi(r,z):=r\phi^\tht(r,z)\quad\mbox{and}\quad  \xi(r,z)=b^\tht(r,z)/r.$$ Then we have for $\Pi:=\{(r,z)\in\mathbb{R}^2\,:\, r>0\},$
 \begin{equation}\label{eq:psi-origin}
	\begin{split}
		\psi(r,z) &=  C_a \int_{\Pi}
		    \left[\int_0^{\pi}  \frac{r\bar{r}\ct}{\left({r^2-2r\bar{r}\ct + \bar{r}^2 + (z-\bar{z})^2}\right)^{(3-2a)/2}}  \ud\tht  \right]  
		  \xi(\bR,\bz) \bR\, \ud \bR \ud \bz\\
		  &=: C_a \int_{\Pi}
		    \left[G_a(r,z,\bR,\bz)  \right]  
		  \xi(\bR,\bz) \bR\, \ud \bR \ud \bz=:\mathcal{G}_a[\xi] .  
	\end{split}
\end{equation}  
Denoting \begin{equation*}
	\begin{split}
		s(r,\bR,z,\bz) := \left(\frac{(r-\bR)^2 + (z-\bz)^2}{r\bR}\right) 
	\end{split}
\end{equation*} gives 
\begin{equation}\label{eq:psi-F}
	\begin{split}
		\psi(r,z)  
		  &= C_a \int_{\Pi}
		    \left[ {(r\bR)^{ a-\frac 1 2}}  \calF_a(s) \right]  
		  \xi(\bR,\bz) \bR\, \ud \bR \ud \bz 
	\end{split}
\end{equation} with
\begin{equation}\label{eq:F}
\begin{split}
	\calF_a(s) & := 
% {(r\bR)^{ \frac 1 2 -a}}	G_a(r,z, \bR,\bz) 
	 \int_0^\pi \frac{\ct}{(2(1-\ct)+s)^\frac{3-2a}{2}} \, \ud \tht> 0\quad\mbox{for}\quad s>0. 
\end{split}
\end{equation}
%Here we set  $\mathcal{F}_a$ by\begin{equation*}\label{eq:G}\begin{split}	G_a(r,z, \bR,\bz) = {(r\bR)^{ a-\frac 1 2}}  \calF_a(s) . \end{split}\end{equation*}
% We also note $G_a$ is symmetric:$$G_a(\bbx,\bar\bbx)=G_a(r,z,\bR,\bz)=G_a(\bR, \bz,r,z)=G_a(\bar\bbx,\bbx).$$
We have the following elementary estimate   on $\calF_a:\mathbb{R}_{>0}\to \mathbb{R}_{>0}$ for $0\leq a<1$;
\begin{equation}\label{est_F}
		\begin{split}
			\calF_a(s) \lesssim_{\tau,a} s^{-\tau}, \qquad \tau\in[1-a,(5/2) -a]. 
		\end{split}
	\end{equation} Indeed, it can be proved by modifying the argument of \cite{FeSv}, which covers the case $a=1$ (the axisymmetric Euler).\\
%	 (cf. When $a=1$, it is well-known that $\tau$ is required to be on	the range $(0,3/2]$ by \cite{FeSv}.) \Red{good to include a proof}\\
	
		 The velocity  $
V=v^re_r+v^ze_z
$ is given by \begin{equation}\label{eq:u}
\begin{split}
	v^r (r,z)= - \frac{1}{r} \rd_{z} \psi (r,z), \qquad 	v^z (r,z)=  \frac{1}{r} \rd_{r} \psi (r,z),
\end{split}
\end{equation} and one may check directly the divergence free condition $\rd_r (r v^r) + \rd_z (r v^z) = 0.$
If we interpret $\xi$ as an axi-symmetric function on $\mathbb{R}^3$ (\textit{i.e.} $\xi(\bbx)=\xi(r,z)$), it solves 
\begin{equation*}\label{eq_xi_tranp}
\begin{split}
\partial_t \xi + V\cdot \nabla_{\mathbf{x}} \xi=0,
\end{split}
\end{equation*} which gives the fact that any $L^p(\mathbb{R}^3)$ norm is preserved: $\|\xi(t)\|_{L^p(\mathbb{R}^3)}=\|\xi_0\|_{L^p(\mathbb{R}^3)},$ for any $t\ge0$.  In addition,
we have   conservation of the impulse:
$$
\int_{\mathbb{R}^3}r^2\xi(t,\bbx) \ud\bbx=\int_{\mathbb{R}^3}r^2\xi_0(\bbx) \ud\bbx.
$$
%\begin{comment}
From \eqref{eq:psi-F} and \eqref{eq:F}, we obtain   expressions  \begin{equation*}\label{eq:ur-G}
\begin{split}
	v^r(r,z)  =-\frac 1 r \partial_z\psi(r,z)=C_a \int_{\Pi}  \frac{\bz-z}{  r  (r\bR)^{\frac32-a}} \calF'(s)\xi(\bR,\bz) \bR\, \ud \bR \ud \bz,  
\end{split}
\end{equation*}\begin{equation}\label{eq:uz-G}
\begin{split}
	v^z(r,z) =\frac 1 r \partial_r\psi(r,z)=C_a  \int_{\Pi} \left(  \frac{2(r-\bR)}{  r  (r\bR)^{\frac32-a}} \calF'(s) + \frac{(r\bR)^{a-\frac12}}{ r^2}\left[(a-\frac 12) \calF(s)-s \calF'(s) \right] \right) \xi(\bR,\bz) \bR\, \ud \bR \ud \bz. 
\end{split}
\end{equation} 
%\end{comment}
 We define the
 energy of the flow \begin{equation}\label{defn_en}
\begin{split}
	E[\xi] &:= 
	 \int_{\mathbb{R}^3} \psi(\bbx) \xi(\bbx) \ud\bbx =
	C_a   \int_{\Pi}\int_{\Pi}
		    \left[ {(r\bR)^{ a-\frac 1 2}}  \calF_a(s) \right]  
		  \xi(\bR,\bz) \bR\, \ud \bR \ud \bz \, \xi(r,z)r \, \ud r \ud z. 
\end{split}
\end{equation} 
 
%We remark that $0\le E<\infty $ for any axi-symmetric  $\xi \in L^1\cap L^\infty(\mathbb{R}^3)$ when $r^2\xi\in  L^1(\mathbb{R}^3)$ (refer to Lemmas \ref{lem_est_energy}, \ref{lem_est_posi}).  
 
\subsection{Variational framework and key proposition}\label{subsec_variation}
%From now on, we always assume the condition on $a$: $$\frac 1 2 < a <1.$$  

We fix some $\frac12<a<1$ and adopt the variational setting of Friedman--Turkington \cite{FT81}: For $\mu>0$,  we consider the \textit{admissible class} \begin{equation*}\label{defn_adm_cla}
\begin{split}
	K_\mu&:=\left\{
\xi\in L^2(\mathbb{R}^3)\quad|\quad \xi=\xi(r,z)\geq 0,\quad \frac 1 2 \int_{\mathbb{R}^3}r^2\xi(\bbx)\ud\bbx=\mu,\quad \int_{\mathbb{R}^3}\xi(\bbx)\ud \bbx\leq 1
	\right\},\end{split}
\end{equation*} and try to maximize the quantity
 \begin{equation*}\label{defn_mod_en}
\begin{split}
	E_2[\xi] &:=E[\xi]-\int_{\mathbb{R}^3}\xi^2 \ud\bbx.
	\end{split}
\end{equation*} 
We define the maximum value and the maximizer set respectively by 
 \begin{equation*}%\label{defn_mod_en}
\begin{split}
	\mathcal{I}_\mu &:=\sup_{\xi\in K_\mu}E_2[\xi], \qquad S_\mu:=\{
	\xi\in K_\mu\quad|\quad E_2[\xi]=I_\mu	
	\}.
	\end{split}
\end{equation*}  
%and $S_\mu$ is the set of maximizers in the class $K_\mu$.

%\begin{proposition}\label{prop_nonem} For each $\mu>0$,  there exists $\zeta\in S_\mu$ satisfying $$ \zeta(x)=\zeta(r,z)=\zeta(r,-z).$$\end{proposition}

We now introduce the key proposition, which implies Theorem \ref{thm_TW_exist}. 
\begin{proposition}\label{prop_prop}
Let $\mu>0$. There exists  $\xi\in S_\mu$ such that there exists a unique pair $(W,\gamma)\in \mathbb{R}^2$ with $W>0$, $\gamma\geq 0$ satisfying  
\begin{equation}\label{rel_xi_psi}
\xi=(\Psi)_+,%,\quad \Psi(\bbx):=\psi(\bbx)-\frac 1 2 W r^2-\gamma,\quad \psi:=\mathcal{G}_a[\xi].
\end{equation}
where  $$\Psi(\bbx):=\psi(\bbx)-\frac 1 2 W r^2-\gamma\quad\mbox{and}\quad  \psi:=\mathcal{G}_a[\xi].$$
Moreover, $\xi$ is non-negative, bounded, Lipshitz continuous, and compactly supported in $\mathbb{R}^3$, and for each $r>0$, $\xi(r,\cdot_z)$ is nonincreasing as a function of $z>0$ and 
$\xi(r,z)=\xi(r,-z)$.
 \end{proposition}
 
\begin{remark}
\begin{enumerate}
\item

 The relation \eqref{rel_xi_psi} implies 
$$
\nabla^{\perp}_{r,z}\Psi\cdot \nabla_{r,z}\xi=0.
$$ Thus  the above proposition with \eqref{eq:u} guarantees the existence of  $\xi$ satisfying
$$v^r\partial_r\xi+(v^z-W)\partial_z\xi=0$$
 As a result,  $\xi(\bbx-Wte_z)$ 
%These propositions give
is 
 a traveling wave solution
 of \eqref{eq_xi_evol} for %which required 
 Theorem \ref{thm_TW_exist}. Lipschitzness of $\xi$  follows from the relation \eqref{rel_xi_psi}
since both $\psi=\mathcal{G}_a[\xi]$ (for $a>1/2$) and the map $s\mapsto s_+$ are Lipschitz.
 \item
 In order to get   more regular $\xi$ than Proposition \ref{prop_prop} does, one has to consider more regular vorticity function (e.g. see \cite{CQZZ_jfa} for 2D generalized SQG equations). Indeed, in the proof of the proposition, we used Lipschitz vorticity $f(s)=s^+$. Instead, we may use smoother $f$ such as $f(s)=(s^+)^{1+\varepsilon}, \varepsilon>0$.
\item 
If one just want to prove existence of traveling waves for $a\leq \frac 1 2 $, then one may follow the approach of \cite{CQZZ_imrn} where the authors proved the existence of traveling dipoles for 2D generalized SQG equations.   However, when $a\leq \frac 1 2$, % it is natural to expect that the stream $\psi$ is less regular than $C^1$. Thus
 it is a difficult question whether the obtained $\xi$ is Lipschitz or not since regularity of the stream $\Phi=(-\Delta)^{-a}B$ may not be Lipschitz. 
\end{enumerate}

 \begin{comment}
\begin{enumerate}
\item 
\item We may use scaling for  $\nu,\lambda\in(0,\infty)$ by setting
$$	E^\lambda_2[\xi]:=E[\xi]-\frac{1}{\lambda}\int_{\mathbb{R}^3}\xi^2 \ud\bbx
$$ and
$$
	K_{\mu,\nu}:=\{
\xi\in L^2(\mathbb{R}^3)\quad|\quad \xi=\xi(r,z)\geq 0,\quad \frac 1 2 \int_{\mathbb{R}^3}r^2\xi(\bbx)\ud\bbx=\mu,\quad \int_{\mathbb{R}^3}\xi(\bbx)\ud \bbx\leq \nu
	\}.
	$$

\item Instead of using $L^2$-norm, we may use  any $p \in(p_0(a),\infty)$ power-norm (for some $p_0(a)>1$):
 $$	E_p[\xi]:=E[\xi]- \frac{1}{p-1}\int_{\mathbb{R}^3}\xi^p \ud\bbx.	
$$ One may prove the existence of a maximizer satisfying 
$$
\xi=\left(\frac{(\Psi)_+}{p/(p-1)}\right)^{1/(p-1)}.
$$
Passing to the limit $p\to\infty$ will give a patch traveling wave as in \cite{FT81}.

\end{enumerate}
\end{comment}
\end{remark}

\subsection{Elementary estimates}
%\Red{Let $ \frac 1 2<  a <1$.}
We first estimate the stream $\psi$ for  $\xi\in   L^{2} \cap L^1 $.
\begin{lemma}
%[Proposition 2.1 in  \cite{AC2019}]
\label{lem_est_stream}
%The estimates 

For   axi-symmetric   $\xi\in \left(   L^{2} \cap L^1\right)(\mathbb{R}^3)$, the   stream function $\psi=\mathcal{G}_a[\xi]$ satisfies  
\begin{equation}\label{est_psi}
|\psi(r,z)|
\lesssim_a  r^{1+a}
   \|\xi\|_{L^{3/(a+1)}(\mathbb{R}^3) } +   r^{2a}  \|\xi\|_{ L^{3/2}(\mathbb{R}^3)},\quad \forall (r,z)\in\Pi.
   \end{equation}
 %  Moreover,\Red{probably we do not need the estimate below. in fact it's meaningless for $a<5/4$} \begin{equation}\label{est_psi_bdd}|\psi(r,z)|\lesssim_{a,\varepsilon} \left(r^{-(5a-4-7\varepsilon)}+r^{-(3-2a)}\right) (   \|\xi\|_{L^{1}\cap L^2(\mathbb{R}^3) } +     \|r^2\xi\|_{ L^1(\mathbb{R}^3)})     \end{equation} for each $(r,z)\in\Pi$ and for each $0<\varepsilon\leq a-\frac 1 2$.  

%\Red{ $1/r^?$ estimate is needed ? will be needed anyway}
\end{lemma}
 
\begin{proof}
  
  %Fix $a\in(1/2,1]$. Then there is $\tau<1/2$ which is valid on $\eqref{est_F}$ for the given $a$.

%Let $p\geq 1$.  
We decompose 
%the representation \eqref{form_psi}  can be written by
\begin{equation*}\label{est_decom_psi}
|\psi(r,z)|\lesssim 
 \int_\Pi
G_a(r,z,r',z') 
|\xi(r',z')|r'dr'dz'
=\int_{t<r/2}+\int_{t\geq r/2}=:I+II,
\end{equation*}
 for 
$  t:%=t(r,z,r',z'):
=\sqrt{(r-r')^2+(z-z')^2}.$
%In particular $F:\mathbb{R}_+\rightarrow \mathbb{R}_+$ is decreasing with the estimate\begin{equation}\label{est_F_sp}F(s)\leq \frac{C_p}{s^{1/2p}}, \quad F(s)\leq \frac{C}{s^{3/2}},\quad s>0\end{equation}   (see \eqref{est_F}). We may assume $\xi\geq 0$.\\We split the integral\begin{equation*}\begin{split}
%\psi(r,z)=\mathcal{G}[\xi](r,z) = \int_\Pi\left(\frac{\sqrt{rr'}}{2\pi}\,F(\frac{t^2}{rr'})\right)\xi(r',z')r'dr'dz'=\int_{t<r/2}+\int_{t\geq r/2}=:(I)+(II).\end{split}\end{equation*}
  For the term $I$,  by using  \eqref{est_F} with 
$\tau=1-a$,  and by taking any $p>1$ satisfying $\tau p <1$,  H\"older's inequality
with $(1/p)+(1/p')=1$
 implies that %\Red{need $2\tau p <2$ i.e. $\tau p <1$}
\begin{equation*}\begin{split}\label{est_stream_1} 
I&\lesssim \int_{t<r/2}
\frac{(rr')^{a-\frac 1 2 +\tau}}{t^{2\tau}}
%{\left(|r-r'|^2+|z-z'|^2\right)^{\frac 1 {2p}}}
 |\xi(r',z')|r' dr'dz' \lesssim  \left(\int_{t<r/2}  
 \frac{(rr')^{p(a-\frac 1 2 +\tau)}}{t^{2\tau p}}
 % {\left(|r-r'|^2+|z-z'|^2\right)^{1/2}}
 r' dr'dz'\right)^{1/p}  \|\xi1_{\{t<r/2\}}\|_{L^{p'}(\mathbb{R}^3)}
 \\ &\lesssim r^{2a-1+\frac 3 p }  \|\xi1_{\{t<r/2\}}\|_{L^{p'}(\mathbb{R}^3)},
\end{split}\end{equation*}  where in the last inequality we used $r\sim r'$.
Thus by setting
$\frac 1 p= 1-a+\varepsilon$ for any $\varepsilon\in(0, a-\frac 1 2]$ 
%  $p=\frac 3 {3-2a}>1$ and taking any  $\tau\in(1-a,\frac {3-2a}{3})$,
   we have 
$$
I \lesssim \ 
 r^{2-a+3\varepsilon  }  \|\xi1_{\{t<r/2\}}\|_{L^{1/(a-\varepsilon)}(\mathbb{R}^3)}.
$$ We can set %$\varepsilon=\frac 1 2 \min\{\frac a 3, a-\frac 1 2\}$, 
$\varepsilon:=\frac 1 3(2a-1)$ so that $0<\varepsilon<\min\{\frac a 3, a-\frac 1 2\}$ whenever $\frac 1 2 <a <1$. It gives the first term in \eqref{est_psi}.

 For the term $II$, we take $
\tau=\frac 5 2 -a $ in \eqref{est_F}. Since $r'\leq 3t$ for $t\geq r/2$, we obtain
%$t\geq r/2$ implies $r'\leq |r'-r|+r\leq 3t$, we estimate
\begin{equation*}\label{est_ii_inter}\begin{split}
II&\lesssim \int_{t\geq r/2}
\frac{(rr')^{2}}{t^{5-2a}}
|\xi(r',z')|r'dr'dz'
\lesssim \int_{t\geq r/2}
\frac{r^2}{t^{3-2a}}
|\xi(r',z')|r'dr'dz' \\
&\lesssim r^2\left(
\int_{t\geq r/2}
\frac{1}{t^{(3-2a)3}}
 r'dr'dz'
\right)^{1/3}\|\xi\|_{L^{3/2}(\mathbb{R}^3)}
\lesssim r^{2a}
\|\xi\|_{L^{3/2}(\mathbb{R}^3)}.
\end{split}\end{equation*} 
We have proved
 \eqref{est_psi}. \end{proof}
 
Now we estimate the energy $E$ for $\xi \in L^1_w\cap L^{2}\cap L^1$, where we say that $\xi$ lies on $ L^1_w$ if
$$\|\xi\|_{L^1_w}:=\|r^2\xi\|_{L^1(\mathbb{R}^3)}=\int_{\mathbb{R}^3}r^2|\xi|\ud \bbx<\infty.$$

\begin{lemma}
%[Proposition 2.1 in  \cite{AC2019}]

\label{lem_est_energy}  For   axi-symmetric   $\xi \in \left(L^1_w\cap L^{2}\cap L^1\right)(\mathbb{R}^3)$, we have %the following  estimates:
\begin{equation}\label{est_E} 
|E[\xi]|\leq E[|\xi|]
\lesssim_a 
 \left(  
   \|\xi\|_{L^{\frac 3 2}(\mathbb{R}^3) } +   \|\xi\|_{ L^{\frac{3}{a+1}}(\mathbb{R}^3)} \right) (\| \xi\|_{L^1_w}+ \|\xi\|_{L^1 (\mathbb{R}^3)}).
\end{equation}
%\begin{equation}\label{est_iint} \left|\int_\Pi\int_\Pi G_a(r,z,r',z')\xi_1(r,z)\xi_2(r',z') rr' dr'dz'drdz\right|\lesssim_a \left(     \|\xi_1\|_{L^{\frac 3 2}(\mathbb{R}^3) } +   \|\xi_1\|_{ L^{\frac{3}{a+1}}(\mathbb{R}^3)} \right) (\|\xi_2\|_{L^1_w}+ \|\xi_2\|_{L^1 (\mathbb{R}^3)}).\end{equation}
%\begin{align}\label{est_E_diff} &\left|E[\xi_1]-E[\xi_2]\right|\lesssim \left(   \|r^2(\xi_1+\xi_2)\|_{1}+   \|\xi_1+\xi_2\|_{L^1\cap L^2} \right) \|r^2(\xi_1-\xi_2)\|^{1/2}_{1} \|\xi_1-\xi_2\|^{1/2}_{1}. \end{align}
%The measure $rdrdz$ is suppressed in \eqref{est_iint}.
 
\end{lemma}
\begin{proof}

We first use \eqref{est_psi} to estimate for   axi-symmetric   $  \xi_1,\xi_2\in \left(L^1_w\cap L^{2}\cap L^1\right)$,
\begin{equation*}\begin{split}
\int_\Pi\left(\int_\Pi G_a(r,z,r'z')|\xi_1(r,z)|rdrdz\right)|\xi_2(r',z')|r'dr'dz' 
&\lesssim
 \left(  % \|r^2\xi\|_{1}+
   \|\xi_1\|_{L^{\frac 3 2} } +   \|\xi_1\|_{ L^{\frac{3}{a+1}}} \right)\int_\Pi [1+ (r')^2]  |\xi_2(r',z')|r'drdz\\
&\lesssim
 \left(  % \|r^2\xi\|_{1}+
   \|\xi_1\|_{L^{\frac 3 2} } +   \|\xi_1\|_{ L^{\frac{3}{a+1}}} \right) (\|\xi_2\|_{L^1_w}+ \|\xi_2\|_{L^1 }).
\end{split}\end{equation*}   
 Then, the estimate \eqref{est_E} follows  by setting $\xi_1\equiv\xi_2$. 
\end{proof}
\begin{comment}
\begin{lemma}\label{lem_est_posi}  For   axi-symmetric   $\xi \in \left(L^1_w\cap L^2\cap L^1\right)(\mathbb{R}^3)$, if
$\xi \in L^\infty(\mathbb{R}^3)$, then
we have %the following  estimates:
\begin{equation}\label{est_E_posi} 
 E[\xi]= C_a\int_{\mathbb{R}^3} |\Gamma^{1/2} B|^2\ud \bbx\geq 0.
\end{equation} 
\end{lemma}
\begin{proof}

  The assumption $\xi \in \left(L^1_w\cap L^\infty\cap L^1\right)(\mathbb{R}^3)$ implies that 
  $B=r \xi(r,z) e_\tht(\tht)$ lies on $L^1\cap L^2(\mathbb{R}^3)$. Indeed, we have 
$$
\int |B|\,\dd\bbx=\int r|\xi |^{1/2}|\xi|^{1/2}\,\dd\bbx\leq
\|r^2\xi\|_{1}^{1/2}  \|\xi\|_{1}^{1/2}<\infty % \sqrt{\int r^2|\xi  dx}\cdot \sqrt{\int|\xi| dx}<\infty
$$
  and 
%$$\int |\omega|\,dx=\int r|\xi|\,dx\leq \sqrt{\int r^2|\xi |dx}\sqrt{\int|\xi| dx}<\infty$$ and 
$$\int |B|^2\,\dd\bbx=\int r^2|\xi|^2\,\dd\bbx\leq \|\xi\|_\infty\int r^2|\xi|\,\dd\bbx <\infty.$$  
   As a consequence of Sobolev embedding, $\Gamma^{1/2} B$ lies on $L^p(\mathbb{R}^3)$ for any
  $$\frac 3 {3-a} < p \leq \frac 6 {3-2a}.$$ In particular, we get
   $\Gamma^{1/2} B\in L^2(\mathbb{R}^3)$. Then, the identity in \eqref{est_E_posi} follows from integration by parts (and by a proper approximation) from \eqref{defn_en}.
   \end{proof}
\end{comment}
In the next lemma, we show regularity of the stream $\Phi$ when $\xi$ is bounded and supported  in $\{r<R\}$ for some $R<\infty$. The result will be used to prove that a maximizer $\xi\in S_\mu$ is Lipschitz and compactly supported.
\begin{lemma}\label{lem_est_psi_con}  Let      $\xi \in \left(L^1_w\cap L^2\cap L^1\right)(\mathbb{R}^3)$ be axi-symmetric. %Then,\begin{enumerate}\item  the stream function $\psi:=\mathcal{G}_a[\xi]$ is continuous on $\mathbb{R}^3$, and \item 
If we assume further 
\begin{equation}\label{assump1}
 \xi\in L^\infty\quad\mbox{and}\quad \mbox{spt}\,(\xi)\subset \{r<R\}\quad\mbox{for some}\quad R<\infty, 
\end{equation}
 then
 \begin{equation}\label{Phi_reg}
\Phi:=(-\lap)^{-a}B \quad \mbox{lies on}\quad  C^{1,\alpha}\cap W^{1,q}(\mathbb{R}^3)
 \end{equation}
%$\psi\in C^{1,\alpha}$
 for some $\alpha=\alpha(a)>0$ and $q=q(a)>1$ with decay
 \begin{equation}\label{Phi_decay}
 \frac{\Phi(\bbx)}{r}= \frac{\Phi(\bbx)}{\sqrt{x_1^2+x_2^2}}\to 0\quad\mbox{as}\quad |\bbx|\to \infty.
 \end{equation} Moreover, the energy of $\xi$ defined by \eqref{defn_en} is nonnegative with the identity
\begin{equation}\label{est_E_posi} 
 E[\xi]= C_a\int_{\mathbb{R}^3} |(-\lap)^{a/2} B|^2\ud \bbx\geq 0.
\end{equation}  
%\end{enumerate}
\end{lemma}
\begin{proof}

The assumption \eqref{assump1} gives $B=B(\bbx)=r\xi(r,z) e_\theta(\theta)\in L^1\cap L^\infty(\mathbb{R}^3)$.
Since %$\Phi=(-Delta)^{-a} B$ and
 $\nabla\Phi=\nabla(-\Delta)^{-a} B$ with $a>1/2$,
 we have  
\begin{equation}\label{est_nab_Phi}
 \nabla \Phi\in L^{q}\cap C^{0,\alpha}(\mathbb{R}^3) 
\end{equation} 
 for any $q>\frac{3}{4-2a}$ and for any $0<\alpha<2a-1$  
%  \Red{any reference for $C^\alpha$?} 
  by Sobolev embedding.
 Similarly, 
 $\Phi=(-\Delta)^{-a} B\in L^{q}(\mathbb{R}^3)$ for any $q >\frac{3}{3-2a}$.  
 % and the Riesz transform.  
 Hence, we get
 \begin{equation*}\label{phi_w_1_q}
 \Phi\in W^{1,q}\cap C^{1,\alpha}(\mathbb{R}^3)\quad\mbox{for any}\quad q>\frac 3 {3-2a} \quad\mbox{and for any} \quad \alpha\in(0,2a-1). \end{equation*}
 
  To show the decay \eqref{Phi_decay},
   we recall \eqref{est_nab_Phi}.
   By Hardy's inequality, $\nabla \Phi\in L^q$  implies 
   $ \Phi/r\in L^q$. On the other hand,
   $ \Phi\in C^{1,\alpha}$ with $\Phi|_{r=0}=0$ implies
   $\Phi/r\in C^{\alpha}$. Hence we obtain the decay \eqref{Phi_decay}.
 
   For the identity \eqref{est_E_posi}, we observe that $\Gamma^{1/2} B$ lies on $L^p(\mathbb{R}^3)$ for any
  $p> \frac 3 {3-a}$. In particular, we get
   $\Gamma^{1/2} B\in L^2(\mathbb{R}^3)$. Then, the identity  follows  from \eqref{defn_en} by   integration by parts:
\begin{align*} E[\xi]&%=\int r\xi\frac \psi r  \dd\bbx
= \int B\cdot \Phi \dd\bbx
=\int(-\Delta)^{-a}B\cdot B \dd\bbx
=\int |(-\Delta)^{-a/2} B|^2 \dd\bbx. \qedhere \end{align*}  
\end{proof}

\subsection{Proof of existence of a maximizer (Proposition \ref{prop_prop})}
% \Red{We fix $\frac 1 2< a <1$,} and any estimates in the sequel may depend on the choice of $a$. 

As we defined $K_\mu, I_\mu,$ and $S_\mu$, we additionally define 
 \begin{equation*}\label{defn_adm_cla_pr}
\begin{split}
	K_\mu'&:=\left\{
\xi\in L^2(\mathbb{R}^3)\quad|\quad \xi=\xi(r,z)\geq 0,\quad \frac 1 2 \int_{\mathbb{R}^3}r^2\xi(\bbx)\ud\bbx\leq \mu,\quad \int_{\mathbb{R}^3}\xi(\bbx)\ud \bbx\leq 1
	\right\},\\
	I_\mu'&:=\sup_{\xi\in K_\mu'}E_2[\xi],\\
	S_\mu'&:=\{
\xi\in K_\mu'\quad|\quad E_2[\xi]=I_\mu'	
	\}.
	\end{split}
\end{equation*} 
%Say $S_\mu'$ be the set of maximizers in $K_\mu'$ for the functional $E_2$.
Since $0\in K_\mu'$, we simply note $K_\mu\subsetneq K_\mu'$ and $I_\mu\leq I_\mu'$. In the sequel, we will prove
$$I_\mu'=I_\mu\in(0,\infty)\quad\mbox{and}\quad S_\mu'=S_\mu\neq \emptyset.$$ %\Red{has to write down}

\begin{lemma}\label{lem_nontrivial_extra} Let $\mu>0$. Then we have
 \begin{equation}\label{est_i_nontrivial}
0<I_{\mu}\leq I_\mu'<\infty,
\end{equation}
and  the value is strictly increasing :
 \begin{equation}\label{est_i_strict_mono}
 I_{\mu_1}<I_{\mu_2}\quad\mbox{for any}\quad 0<\mu_1<\mu_2<\infty. 
 \end{equation}
%In addition, \begin{equation}\label{est_L2_bdd} \sup_{\xi\in K_\mu}\|\xi\|_{L^2(\mathbb{R}^3)}<\infty.\end{equation}

\end{lemma}

\vspace{5pt}

\begin{proof}
 
  By \eqref{est_E} and Young's inequality, we have, for any $\xi\in K'_{\mu}$, %\Red{  $a\geq 1/2$ is used so that $\frac 3{a+1}<2$}
  \begin{equation}\begin{split}\label{est_e_by_l2}
E[\xi]&\leq  C\left(  
   \|\xi\|_{L^{1}(\mathbb{R}^3) } +   \|\xi\|_{ L^{\frac{3}{a+1}}(\mathbb{R}^3)} \right) (\| \xi\|_{L^1_w}+ \|\xi\|_{L^1 (\mathbb{R}^3)}) 
   \leq C(\mu+1)+ C(\mu+1)\|\xi\|_{ L^{\frac{3}{a+1}}(\mathbb{R}^3)}  
   \\&   \leq C(\mu+1)+ C(\mu+1)\|\xi\|_{ L^{2}(\mathbb{R}^3)} 
 \leq C(\mu+1)+ C(\mu+1)^2+\frac 1 2\|\xi\|^2_{ L^{2}(\mathbb{R}^3)} .
\end{split}  \end{equation} It implies
 $$
  E_2[\xi]=E[\xi]-\|\xi\|^2_{L^2(\mathbb{R}^3)}\leq   C(\mu+1)+ C(\mu+1)^2+\frac 1 2\|\xi\|^2_{ L^{2}(\mathbb{R}^3)}-\|\xi\|^2_{L^2(\mathbb{R}^3)}\leq C(\mu+1)^2<\infty,
$$
  which gives $I_\mu'<\infty$.

To prove $I_\mu>0$, we set $\xi_1=1_{B_a(0)}$ for $B_a(0)=\{x\in \mathbb{R}^{3}\ |\ |x|<a\}$ and choose $a>0$ so that $\frac  1 2 \int r^2\xi_1(x)\dd \bbx=\mu$. We set $\xi_{\sigma}(x)=\sigma^{5}\xi_1(\sigma x)$ for $0<\sigma<1$, and observe that 
\begin{align*}
&\frac 1 2 \int_{\mathbb{R}^{3}}r^2\xi_\sigma(x)\dd \bbx=\frac 1 2 \int_{\mathbb{R}^{3}}r^2\xi_1(x)\dd \bbx=\mu, \\
&\int_{\mathbb{R}^{3}}\xi_\sigma\dd \bbx=\sigma^2 \int_{\mathbb{R}^{3}}\xi_1\dd \bbx\leq 1, \\
%&\int_{\mathbb{R}^{3}}\xi_\sigma^2\dd \bbx=\sigma^7 \int_{\mathbb{R}^{3}}\xi_1^2\dd \bbx, \\
%&\psi_\sigma(r,z)=\sigma^{3-2a}\psi_1(\sigma r,\sigma z),\\
%&E[\xi_\sigma]=\sigma^{5-2a} E[\xi_1],\\
&E_2[\xi_{\sigma}]=
\sigma^{5-2a} \left(E[\xi_1]-\sigma^{2+2a}\int \xi_1^2 \dd \bbx\right).
\end{align*}

Thanks to $\xi\geq 0$, we have $\psi\geq 0$. Then the definition of the energy \eqref{defn_en}  gives  $E[\xi_1]>0$. Thus
 for sufficiently small $0<\sigma\ll 1$,  $\xi_{\sigma}\in K_{\mu} $  satisfies
\begin{align*}
I_\mu\geq E_{2}[\xi_{\sigma}]>0, 
\end{align*}
which proved  \eqref{est_i_nontrivial}. 

For   monotonicity in $\mu>0$, we fix any $0<\sigma<1$, take any $\xi\in K_\mu$ and set
$\xi^{\sigma}(x)=\sigma^{\frac 5 2 + a}\xi(\sigma x)$. It gives
\begin{align*}
&\frac 1 2 \int_{\mathbb{R}^{3}}r^2\xi^\sigma(x)\dd \bbx =\sigma^{-\frac 5 2 +a}\mu>\mu, \\
&\int_{\mathbb{R}^{3}}\xi^\sigma\dd \bbx\leq \sigma^{a-\frac 1 2}\leq 1, %\Red{\mbox{need }  a\geq \frac1 2 } 
\\
&\int_{\mathbb{R}^{3}}(\xi^\sigma)^2\dd \bbx=\sigma^{2+2a} \int_{\mathbb{R}^{3}}(\xi)^2\dd \bbx, \\
%&\psi^\sigma(r,z)=\sigma^{\frac 1 2 -a}\psi(\sigma r,\sigma z),\\
&E[\xi^\sigma]=  E[\xi].
%&E_2[\xi_{\sigma}]=\sigma^{7-2a} \left(E[\xi_1]-\sigma^{2+2a}\int \xi_1^2 \dd x\right).
\end{align*} So   we have $\xi^\sigma\in K_{\sigma^{-\frac 5 2 + a}\mu}$ and 
\begin{equation}\begin{split}\label{est_scal_mon}
I_{\sigma^{-\frac 5 2 +a}\mu}\geq E_2[\xi^\sigma]&=E[\xi]-\sigma^{2+2a}\|\xi\|^2_{L^2(\mathbb{R}^3)}
 =E_2[\xi]+\underbrace{(1-\sigma^{2+2a})}_{>0}\|\xi\|^2_{L^2(\mathbb{R}^3)}>E_2[\xi].
\end{split}\end{equation} By taking $\sup_{\xi\in L_\mu}$, we get
$$
I_{\sigma^{-\frac 5 2 +a}\mu}\geq I_\mu.
$$ For strictness in the monotonicity \eqref{est_i_strict_mono}, let's assume 
\begin{equation}\label{hyp_mon}
I_{\sigma^{-\frac 5 2 +a}\mu}= I_\mu
\end{equation}
 for a contradiction. Then take a sequence $\{\xi_n\}_{n=1}^\infty\subset K_\mu$ satisfying $E_2[\xi_n]\to I_\mu$. This hypothesis \eqref{hyp_mon} together with \eqref{est_scal_mon} implies
$$
\|\xi_n\|_{L^2(\mathbb{R}^3)}\to 0.
$$
The estimate \eqref{est_E} of Lemma \ref{lem_est_energy} gives
\begin{equation*} \begin{split}
|E[\xi_n]|&\leq C
 \left(  
   \|\xi_n\|_{L^{\frac 3 2}(\mathbb{R}^3) } +   \|\xi_n\|_{ L^{\frac{3}{a+1}}(\mathbb{R}^3)} \right) (\| \xi_n\|_{L^1_w}+ \|\xi_n\|_{L^1 (\mathbb{R}^3)})\\
   &\leq C(\mu+1)
 \left(  
   \|\xi_n\|^{\frac 1 3}_{L^{1}(\mathbb{R}^3) } \|\xi_n\|^{\frac 2 3}_{L^{2}(\mathbb{R}^3) } +     \|\xi_n\|^{\frac{2a-1}{3}}_{L^{1}(\mathbb{R}^3) } \|\xi_n\|^{\frac{4-2a}{3}}_{L^{2}(\mathbb{R}^3) }  \right)\quad   \to\quad  0\quad \mbox{as}\quad n\to\infty. %\Red{\mbox{need }  a\geq \frac 1 2} 
\end{split}\end{equation*} These two convergences mean $I_\mu=\lim_{n\to\infty}E_2[\xi_n]= 0$, which contradicts the fact $I_\mu>0$ from \eqref{est_i_nontrivial}. As a result, we obtain the strict monotonicity \eqref{est_i_strict_mono}. \end{proof}
 
 We observe that any maximizing sequence is bounded in $L^2$.
\begin{lemma}\label{lem_l2_unif}
 Let $\mu>0$. Then any   sequence $\{\xi_n\}_{n=1}^\infty\subset K_\mu$ satisfying $E_{2}[\xi_n]\to I_{\mu}$ satisfies
 % is uniformly bounded in $L^{2}$.
 $$\limsup_{n\to\infty} \|\xi_n\|_{L^2(\mathbb{R}^3)}<\infty.$$
 \end{lemma}
 \begin{proof} 
 
By using the estimate \eqref{est_e_by_l2}, we have  \begin{align*}
\|\xi_n\|^2_{L^2(\mathbb{R}^3)}=E[\xi_n]-E_2[\xi_n] \leq    C(\mu+1)^2+\frac 1 2\|\xi_n\|^2_{ L^{2}(\mathbb{R}^3)}- E_2[\xi_n] , 
\end{align*} which implies
$$
\limsup_{n\to\infty}\|\xi_n\|^2_{L^2(\mathbb{R}^3)}\leq C(\mu+1)^2-I_\mu\leq C(\mu+1)^2 <\infty. \qedhere 
$$

\end{proof}

The following lemma is useful when we need convergence of the  energy for a weak-convergent sequence $\{\xi_n\}$ when the energy of each member $\xi_n$ is   uniformly concentrated in a fixed bounded set.

\begin{lemma}\label{lem_en_diff}For  non-negative axi-symmetric functions $\xi_1, \xi_2\in \left(L^{1}_w\cap L^2\cap L^{1}\right)(\mathbb{R}^{3})$ and for   axi-symmetric  set $U\subset\mathbb{R}^3$,  we have
\begin{equation}\begin{split}\label{est_en_diff}
\left| E[\xi_1]-E[\xi_2]\right|
&\leq C_a\left| \int_{{  U } }\int_{{  U}  }  {G_a}(\bbx,\bar\bbx)\Big(\xi_1(\bbx)\xi_1(\bar\bbx)-\xi_2(\bbx)\xi_2(\bar\bbx)\Big)\,d\bbx d\bar\bbx \right|
\\&\quad \quad + 
C_a \int_{\mathbb{R}^{3}\setminus {U }}\xi_1(\bbx)\mathcal{G}_a[\xi_1](\bbx)\,d\bbx 
+  C_a\int_{\mathbb{R}^{3}\setminus {U }}\xi_2(\bbx)\mathcal{G}_a[\xi_2](\bbx)\,d\bbx.
\end{split}\end{equation} 
\end{lemma}
\begin{proof}
The proof is exactly parallel to  that of Lemma 4.5 in \cite{Choi_hill} that is the case $a=1$.
\end{proof}

% its logarithm behavior \eqref{log_beha} near $x=y$.  

 For given non-negative axi-symmetric $\xi$, we can define  the symmetrical rearrangement $\xi^*(r,\cdot_z)$ of $\xi(r,\cdot_z)$ for each $r>0$ about the plane $z=0$. 
 Then the  Riesz rearrangement inequality implies the energy inequality below (\textit{e.g.} see \cite[Section 3.3]{LiebLoss}).
\begin{lemma}[Steiner symmetrization]\label{lem_steiner}
For any non-negative axi-symmetric $\xi\in L^{2}\cap L^1_w\cap L^{1}(\mathbb{R}^{3} )$,   there exists a non-negative axi-symmetric $\xi^{*} \in L^{2}\cap L^1_w\cap L^{1}(\mathbb{R}^{3} )$ such that 
\begin{equation}\label{cond_sym}
\begin{aligned}
&\xi^{*}(r,z)=\xi^{*}(r,-z)\quad \mbox{for any}\quad  r,z>0\quad\mbox{and for each}\quad r>0,\quad \xi^{*}(r,\cdot_z)\ \textrm{is non-increasing for}\ z>0. 
\end{aligned}
 \end{equation}
Moreover, 
\begin{align*}
&||\xi^{*}||_{L^q(\mathbb{R}^3)}=||\xi||_{L^q(\mathbb{R}^3)}\quad 1\leq q\leq 2,  \quad ||r^2\xi^{*}||_{L^1(\mathbb{R}^3)}=||r^2\xi||_{L^1(\mathbb{R}^3)},  \quad\mbox{and}\quad E[\xi^*]\geq E[\xi]. 
\end{align*}
\end{lemma}

The following lemma says that the kinetic energy is concentrated in a bounded domain when $\xi$ satisfies  {the monotonicity } condition \eqref{cond_sym}.  The proof is parallel to that of Lemma 4.9 if \cite{Choi_hill}.
\begin{lemma}\label{lem_stream_AR}
Let $\xi\in (L^1_w \cap L^{2}\cap L^{1})  (\mathbb{R}^{3})$ be an  axi-symmetric nonnegative function satisfying 
$\xi=\xi^*$. Then we have % \Red{should be modified}
\begin{equation}\label{est_stream_AR}%\color{blue}
\begin{aligned}
\int_{\mathbb{R}^{3}\backslash Q} \xi\mathcal{G}_a[\xi] \ud x 
\lesssim  \left(\frac{1}{\sqrt A}  +
 \frac{1}{R^{1-a}}   \right)  \left(\|\xi\|_{L^1\cap L^2}  +\|r^2\xi\|_{1} \right)^2, 
\end{aligned}
\end{equation} where
$$Q=Q_{A,R}=\{x\in\mathbb{R}^3\, |\,  \quad |z|<AR, \quad r <R\}$$
 provided 
$R\geq 1$ and $A\geq2$. %; the constant $C$ is a universal constant. 
\end{lemma}
\begin{comment}
    \begin{proof}%[Proof of Lemma \ref{lem_stream_AR}]
       
We decompose

\begin{align*}
\int_{\mathbb{R}^{3}\backslash Q }\psi \xi \dd x
=\int_{r\geq R}+\int_{\substack{r<R, \\ |z|\geq AR}}=:I+II ,
\end{align*}\\
and estimate, by \eqref{est_psi}, % with $\delta=1$,
\begin{align*}
%\int_{r\geq R}\psi \xi \dd x
I\lesssim 
\left(\|\xi\|_{L^1\cap L^2}   \right) \int_{r\geq R}{(r^{1+a}+r^{2a})} \xi\dd x
\lesssim 
\left(\|\xi\|_{L^1\cap L^2}   \right) \int_{r\geq R}{(r^{1+a}+1)} \xi\dd x
\lesssim \left(\frac 1 {R^{1-a}}+\frac{1}{R^{2}}\right)\cdot\|\xi\|_{L^1\cap L^2}  \cdot \|r^2\xi\|_{1}.
\end{align*}\\
For $II$, since  $r<R$ and $|z|\geq AR$  imply $r\leq |z|/A$, applying \eqref{est_stream_sym} yields 

\begin{align*}
%\int\limits_{\substack{r<R, \\ |z|\geq AR}} \psi \xi\dd x
II
&\lesssim \left(\|\xi\|_{L^1\cap L^2}  +\|r^2\xi\|_{1} \right) \int\limits_{\substack{r<R, \\ |z|\geq AR}} 
\left(\frac{r^{1+a}}{\sqrt A} +
\frac{1}{  A} +
 r^2\left(\frac{A}{|z|}\right)^{5-2a}  \right)
 \xi \dd x \\
 &\lesssim \left(\|\xi\|_{L^1\cap L^2}  +\|r^2\xi\|_{1} \right) \int 
\left(\frac{r^{1+a}}{\sqrt A} +
\frac{1}{  A} +
 \frac{r^2}{R^{5-2a}}   \right)
 \xi \dd x \\
  &\lesssim \left(\frac{1}{\sqrt A}  +
 \frac{1}{R^{5-2a}}   \right)\cdot \left(\|\xi\|_{L^1\cap L^2}  +\|r^2\xi\|_{1} \right)^2.
\end{align*} Combining the above estimates, we obtain the conclusion \eqref{est_stream_AR}.
 
  \end{proof}
  \end{comment}

The following lemma ensures  convergence of the energy for any bounded  sequence in $(L^1_w \cap L^{2}\cap L^{1})$ satisfying {the monotonicity} condition \eqref{cond_sym} and converging weakly.

%lem3.5
\begin{lemma}\label{lem_energy_conv}
Let $\{\xi_n\}_{n=1}^\infty$ be a sequence of axi-symmetric non-negative functions on $\mathbb{R}^3$ such that 
\begin{align*}
&\xi_n=(\xi_n)^*\quad\mbox{for each}\quad n,\qquad \sup_{n  }\left\{||\xi_n||_{L^1\cap L^2}+||r^2\xi_n||_{1}  \right\}<\infty,\quad\mbox{and}\\
&\xi_n\rightharpoonup \xi\quad \textrm{in}\ L^{2}(\mathbb{R}^{3})\quad \textrm{as}\quad n\to\infty\quad \mbox{for some non-negative axi-symmetric}\quad \xi\in L^{2}(\mathbb{R}^{3}).
\end{align*} 
 Then we have convergence of  the energy: $E[\xi_{n}]\to E[\xi]$ as $n\to\infty.$
\end{lemma}
   \begin{proof}%[Proof of Lemma \ref{lem_energy_conv}]
 %  \Red{exactly same proof}\\
  First, we observe, by the weak convergence $\xi_n\rightharpoonup \xi$ in $L^{2}(\mathbb{R}^{3})$, % and by the non-negativity of functions $\xi_n$, $\xi$,
$$  \|\xi\|_{L^1\cap L^2}+\|r^2\xi\|_{1}\leq C\quad \mbox{for some }\, C>0.$$
We set a bounded domain
$$Q=Q_{A,R}=\{x\in\mathbb{R}^3\, |\,  \quad |z|<AR, \quad r <R\}.$$
  for $R\geq 1$ and $A\geq1$.
  Then, by \eqref{est_en_diff} of Lemma \ref{lem_en_diff}, we have
  \begin{equation*}\begin{split}%\label{est_en_diff}
\left| E[\xi_n]-E[\xi]\right|
&\leq  \frac{1}{4\pi}\left| \int_{{  Q } }\int_{{  Q}  }  {G}_a(x,y)\Big(\xi_n(x)\xi_n(y)-\xi(x)\xi(y)\Big)\,dxdy \right|
+ \int_{\mathbb{R}^{3}\setminus {Q }}\xi_n\mathcal{G}_a[\xi_n]\,dx 
+  \int_{\mathbb{R}^{3}\setminus {Q }}\xi\mathcal{G}_a[\xi]\,dx.
\end{split}\end{equation*}
Since  $\xi_n$  satisfies {the monotonicity }condition \eqref{cond_sym} for each $n\geq 1$, so does $\xi$. Thus
   we can estimate, by \eqref{est_stream_AR} of Lemma \ref{lem_stream_AR},
   \begin{equation*} 
\begin{aligned}
&\int_{\mathbb{R}^{3}\backslash Q}\xi\mathcal{G}_a[\xi]\ud x 
\lesssim  \left(\frac{1}{\sqrt A}  +
 \frac{1}{R^{1-a}}   \right)\quad\mbox{and}\quad \sup_n\int_{\mathbb{R}^{3}\backslash Q}\xi_n\mathcal{G}_a[\xi_n]\ud x 
\lesssim \left(\frac{1}{\sqrt A}  +
 \frac{1}{R^{1-a}}   \right).
\end{aligned}
\end{equation*}
On the other hand, we can simply check that the kernel $G_a(\bbx,\bar\bbx)=G_a(r,z,\bR,\bz)$ is locally square integrable
$$G_a\in L^2_{loc}(\mathbb{R}^3\times\mathbb{R}^3)$$ 
due to
the estimate \eqref{est_F}. Thus thanks to the weak convergence
%Since $G_a(x,y)\in L^{2}(Q\times Q)$ by Lemma \ref{lem_g_l_2}  and
 $\xi_n(x)\xi_n(y)\rightharpoonup\xi(x)\xi(y)$ in $L^{2}(Q\times Q)$, sending $n\to\infty$ and $A,R\to\infty$ imply  convergence of the energy. \end{proof}
Finally, we show the existence of a maximizer   in $K_\mu'$.
  \begin{lemma}\label{lem_exist}
  Let $0<\mu<\infty$. Then
 $\mathcal{S}'_{\mu}\neq \emptyset.$
\end{lemma}
\begin{proof}%[Proof of Lemma \ref{lem_exist}]
%By the scaling \eqref{scaling}, we reduce to the case $0<\mu<\infty$, $\nu=\lambda=1$. 
Let $\{\xi_n\}\subset K'_{\mu}$ be a   sequence  satisfying $E_2[\xi_n]\to \mathcal{I}'_\mu$. By the Steiner symmetrization (Proposition \ref{lem_steiner}), we 
obtain the corresponding sequence $\{\xi_n^*\}$ in $K'_\mu$ satisfying {the monotonicity }condition \eqref{cond_sym}.
%may assume that $\xi_n$ satisfies satisfies {the monotonicity }condition \eqref{cond_sym}.
By the energy inequality, we know
$$
I_\mu'\geq E_2[\xi_n^*]=E[\xi_n^*]-\|\xi_n^*\|_{L^2(\mathbb{R})^3}^2
\geq  E[\xi_n]-\|\xi_n\|_{L^2(\mathbb{R})^3}^2=E_2[\xi_n].
$$ Then, taking the limit $n\to \infty$, we get
$E_2[\xi_n^*]\to \mathcal{I}'_\mu$.
 Since $\{\xi^*_n\}$ is uniformly bounded in $L^{2}(\mathbb{R}^3)$ by Lemma \ref{lem_l2_unif}, % Remark \ref{rem_interpol}, 
by choosing a subsequence if necessary  (still denoted by $\{\xi^*_n\}$ for simplicity), there exists a non-negative axi-symmetric function $\xi=\xi^*\in L^{2}(\mathbb{R}^3)$ satisfying
%{the monotonicity }condition \eqref{cond_sym} with 
$\xi_n^*\rightharpoonup \xi$ in $L^{2}$. % and $||\omega||_2\leq \liminf_{n\to\infty}||\omega_n||_2$. 
%By the uniqueness of a weak-limit, the function  $\xi$ is axi-symmetric. 
%Moreover it   satisfies {the monotonicity }condition \eqref{cond_sym}. 
Hence
 %Since $\{\xi_n\}$ satisfies all the assumptions of Lemma \ref{lem_energy_conv},
  we can apply Lemma \ref{lem_energy_conv} for $\{\xi^*_n\}$ to get 
  \begin{align*}
\lim_{n\to \infty}E[\xi^*_n]=E[\xi].
\end{align*} 
  %the convergence
  Since the weak-limit $\xi$ lies on $K'_{\mu}$, we know
$E_2[\xi]\leq I'_\mu.$ On the other hand, by taking $\limsup_{n\to\infty}$ to $$E_2[\xi_n^*]=
E[\xi_n^*]-\|\xi_n^*\|_{L^2(\mathbb{R}^3)}^2
,$$ we get
$$
I_\mu'=
E[\xi]-\liminf_{n\to\infty}\|\xi_n^*\|_{L^2(\mathbb{R}^3)}^2
\leq E[\xi]-\|\xi\|_{L^2(\mathbb{R}^3)}^2=E_2[\xi].
$$ It proves $\xi\in S_\mu'$.  
\end{proof}

The next task is to show that the maximizers in $K_\mu'$  lies on the smaller class $K_\mu$. In other words, we claim
$$\emptyset\neq S_\mu'\subset S_\mu,$$ which implies 
$I'_\mu=I_\mu$ (so $S_\mu'= S_\mu$).

 \begin{lemma}\label{lem_same_max}
  Let $0<\mu<\infty$. Then
 $${S}'_{\mu}={S}_{\mu}\neq \emptyset\quad\mbox{and}\quad I'_\mu=I_\mu.$$
\end{lemma}
\begin{proof} 
The proof is parallel to the case $a=1$ in \cite{Choi_hill}, which we sketch now. Let us take any  $\xi\in \mathcal{S}'_{\mu}$ by recalling
$\mathcal{S}'_{\mu}\neq \emptyset$ from Lemma \ref{lem_exist}. We claim  
\begin{equation}\label{claim_iden_tilde}
  \xi\in K_{\mu}.
\end{equation}
 For a contradiction,  % n other words, we need to show $$\frac 1 2 \int r^2 \xi =\mu.$$
%For a contradiction, 
we suppose 
\begin{equation*}\label{assum1} 
\mu>\frac 1 2 \int_{\mathbb{R}^3} r^2 \xi \dd \bbx=:\mu_0,\end{equation*} i.e. we assume $\xi \in K'_\mu\setminus K_\mu$.   
From ${I}'_\mu>0$ by \eqref{est_i_nontrivial} in  Lemma \ref{lem_nontrivial_extra}, we know
$\xi\nequiv 0$ so 
$\mu_0>0$. %Then we get $$\xi\in \mathcal{S}'_{\mu_0}$$ thanks to $\mathcal{P}'_{\mu_0}\subset \mathcal{P}''_{\mu}$.\\
 If we define
 the (relative) translation ${\xi_\tau}$ of $\xi$ away from $z-$axis  for $\tau>0$
  by
\begin{equation*}\label{tau_trans}
 \xi_\tau(r,z)=\begin{cases}&\frac{r-\tau}{r}\xi(r-\tau,z)\quad\mbox{for}\quad r\geq\tau,\\
&0\quad\mbox{for}\quad 0< r<\tau,
\end{cases}\end{equation*} then we have
% \begin{equation}\begin{split}\label{trans_comp}  &\| \xi_\tau\|_1=2\pi\int_{\Pi} r \xi_\tau drdz=2\pi\int_{r\geq \tau} (r-\tau) \xi(r-\tau,z) drdz=2\pi\int_{\Pi} r \xi drdz=\|\xi\|_1\leq 1,\\&\frac 1 2 \int r^2\xi_\tau\,\dd\bbx =\pi   \int_{r\geq\tau} (r-\tau) r^2  \xi(r-\tau,z) drdz=\pi   \int_{\Pi}   (r+\tau)^2 r \xi drdz=\mu_0+\frac 1 2    \int  (2\tau r+\tau^2)  \xi \dd\bbx.\end{split}\end{equation} Thus  we can take 
some constant $\tau>0$ such that %$ \xi_\tau$ defined above satisfies 
  $ \frac 1 2 \int r^2\xi_\tau\,dx=\mu,$ 
\textit{i.e.}  we have $ \xi_\tau\in K_\mu$.  On the other hand, we observe $$E[ \xi_\tau]> E[\xi]$$  by exploiting  the form 
\eqref{eq:psi-origin} of the kernel $G_a$.
\begin{comment} Indeed, we observe, for $(r,z),(r',z')\in\Pi$,
 \begin{equation}\begin{split}\label{exploit} 
 G_a(r+\tau,z,r'+\tau,z') &=
 {\left((r+\tau)(r'+\tau)\right)^{a-\frac 1 2}}\,F_a\left(\frac{(r-r')^2+(z-z')^2}{(r+\tau)(r'+\tau)}\right)\\
 &
 >  ({rr'})^{a-\frac 1 2}\,F_a\left(\frac{(r-r')^2+(z-z')^2}{rr'}\right)=G_a(r,z,r',z')
 \end{split}\end{equation} %\Red{$a>\frac 1 2 $ is used }
  because
$F(\cdot)$ is strictly decreasing due to $F_a'<0
 $.   Thus, we have
 \begin{equation}\begin{split}\label{exploit2}& E[ \xi_\tau]
 =C_a\iint_{\Pi^2} rr'G_a(r,z,r',z')\xi_\tau(r',z')\xi_\tau(r,z)dr'dz'drdz\\
 &
 =C_a\iint_{r>\tau,\, r'>\tau}  (r-\tau)(r'-\tau)G_a(r,z,r',z')\xi (r'-\tau,z') \xi (r-\tau,z)dr'dz'drdz\\
  &
 =C_a\iint_{\Pi^2} rr'G_a(r+\tau,z,r'+\tau,z')\xi (r',z')\xi (r,z)dr'dz'drdz\\
  & >C_a\iint rr'G(r,z,r',z')\xi(r',z')\xi(r,z)dr'dz'drdz=  E[ \xi],
\end{split}\end{equation} where the last inequality comes from
 \eqref{exploit} and non-triviality of $\xi\geq 0$. % \nequiv 0$.
% \color{black}
\end{comment}
Hence we get 
$$ I_\mu\geq E[ {\xi}_\tau]>E[\xi]=I'_\mu,$$ which contradicts to
   $I_{\mu}\leq  I'_{\mu}$ obtained from 
\eqref{est_i_nontrivial}. Hence  $\mu_0=\mu$ so we get
the claim \eqref{claim_iden_tilde}.   
   % $\xi$ should lie on $\mathcal{P}'_{\mu}$. 
As a consequence, we get
 $  \mathcal{I}_{\mu}=\mathcal{I}'_{\mu}$ and 
  $   \mathcal{S}_{\mu}= \mathcal{S}'_{\mu}.$ \end{proof}

Now we are ready to prove Proposition \ref{prop_prop}, which implies Theorem \ref{thm_TW_exist}. 

\begin{proof}[Proof of Proposition \ref{prop_prop}]
We first prove the relation \eqref{rel_xi_psi}. We present a proof for completeness while the proof is essentially  parallel to that of 
\cite[Proposition 2.5]{AC_lamb}.\\

We take any maximizer $\xi\in S_\mu$ whose construction  is guaranteed by Lemmas \ref{lem_same_max} and \ref{lem_exist}.
We note that for each $r>0$, $\xi(r,\cdot_z)$ is nonincreasing as a function of $z>0$ and 
$\xi(r,z)=\xi(r,-z)$.
 Since $\xi\geq0$ is non-trivial, we can take $\delta_0>0$ satisfying $|\{\bbx\in \mathbb{R}^{3}\ |\ \xi\geq \delta_0\}|>0,$ where $|\cdot|$ denotes the Lebesgue measure in $\mathbb{R}^3$.
 Then we take compactly supported axi-symmetric $h_1,h_2\in L^{\infty}(\mathbb{R}^3)$ such that $\textrm{spt}\ h_i\subset \{\xi\geq \delta_0\}$, $i=1,2$,  
\begin{align*}
\int_{\mathbb{R}^{3}}h_1(\bbx)\dd \bbx=1,\quad \int_{\mathbb{R}^{3}}r^2h_1(\bbx)\dd \bbx=0, \quad \mbox{and} \quad \int_{\mathbb{R}^{3}}h_2(\bbx)\dd \bbx=0,\quad \int_{\mathbb{R}^{3}}r^2h_2(\bbx)\dd \bbx=1.
\end{align*}
Consider any 
 $\delta\in (0,\delta_0)$ and compactly supported axi-symmetric $h\in L^{\infty}(\mathbb{R}^{3})$ such that $h\geq 0$ on $\{0\leq \xi\leq \delta\}$. We set
\begin{align*}
\eta:=h-\left(\int_{\mathbb{R}^{3}}h(\bbx) \dd \bbx  \right)h_1-\left(\int_{\mathbb{R}^{3}}r^2 h(\bbx) \dd \bbx \right)h_2
\end{align*}\\
so that $\int_{\mathbb{R}^{3}} \eta(\bbx)\dd \bbx=0$ and $\int_{\mathbb{R}^{3}} r^2\eta(\bbx)\dd \bbx=0$. Observe that $\xi+\varepsilon\eta\geq \delta-\varepsilon||\eta||_{L^\infty}\geq 0$ on $\{\xi\geq \delta\}$ for small $\varepsilon>0$. Since $\eta=h\geq 0$ on $\{0\leq \xi\leq \delta\}$, $\xi+\varepsilon\eta\geq 0$ on $\{0\leq \xi\leq \delta\}$. Hence $\xi+\varepsilon\eta\in K_{\mu}$  for small $\varepsilon>0$. Thus we have
$$
\frac{E_2[\xi]-E_2[\xi+\varepsilon\eta]}{\varepsilon}\geq 0\quad \mbox{for any small} \quad \varepsilon>0.
$$
By using the definition of $E_2$ and by taking the limit on $\varepsilon\searrow  0$, we obtain
$$
\int_{\mathbb{R}^3}(\psi-\xi)\eta\dd \bbx\leq 0.
$$ %\Red{correct energy is coeeficient constant 1}
By the definition of $\eta$, we get
$$
\int_{\mathbb{R}^3}(\psi-\xi)h\dd\bbx -\left(\int_{\mathbb{R}^{3}}h(\bbx) \dd \bbx  \right)\int_{\mathbb{R}^3}(\psi-\xi)h_1\dd\bbx-\left(\int_{\mathbb{R}^{3}}r^2 h(\bbx) \dd \bbx \right)\int_{\mathbb{R}^3}(\psi-\xi)h_2\dd \bbx\leq 0.
$$
By setting $\gamma:=\int_{\mathbb{R}^3}(\psi-\xi)h_1\dd\bbx$ and $\frac 1 2 W:=\int_{\mathbb{R}^3}(\psi-\xi)h_2\dd\bbx$, we get
$$
0\geq 
%\int_{\mathbb{R}^3}(\psi-\xi)h\dd\bbx -\left(\int_{\mathbb{R}^{3}}h(\bbx) \dd \bbx  \right)\gamma-\left(\int_{\mathbb{R}^{3}}r^2 h(\bbx) \dd \bbx \right)\frac 1 2 W=
\int_{\mathbb{R}^3}(\psi-\frac 1 2 W r^2-\gamma-\xi)h\dd\bbx 
=\int_{\{0\leq\xi\leq \delta\}}\dots\dd\bbx+\int_{\{\xi>\delta\}}\dots\dd\bbx.
$$
 We set $\Psi:=\psi-\frac 1 2Wr^2-\gamma$. Since $h$ is an arbitrary function satisfying $h\geq 0$ on $\{0\leq \xi\leq \delta\}$,

\begin{equation*}
\begin{aligned}
\Psi-\xi&=0\quad \textrm{on}\ \{\xi>\delta\},\\
\Psi-\xi&\leq 0\quad \textrm{on}\ \{0\leq \xi\leq \delta\}. 
\end{aligned}
\end{equation*}\\
Since $\delta>0$ is arbitrary, sending $\delta\to 0$ implies 

\begin{equation*}
\begin{aligned}
\Psi-\xi&=0\quad \textrm{on}\ \{\xi>0\},\\
\Psi&\leq 0\quad \textrm{on}\ \{\xi=0\}.
\end{aligned}
\end{equation*}It implies   $\xi=\Psi_+=(\psi-\frac 1 2Wr^2-\gamma)_+$.\\

To show $W>0$, we first compute
 \begin{equation*}\begin{split}
0=-\int_{\{\xi>0\}}\xi\partial_r\xi\dd r\dd z  =-\int_{\{\xi>0\}}\Psi\partial_r\xi\dd r \dd z= \int_{\{\xi>0\}}(\partial_r\Psi)\xi\dd r \dd z=
 \int_{\{\xi>0\}}(\partial_r\psi)\xi\dd r \dd z-
  \int_{\{\xi>0\}}(Wr)\xi\dd r \dd z
\end{split} \end{equation*} so that we get
$$
 W\int_{\mathbb{R}^3}\xi\dd \bbx
 %= \int_{\{\xi>0\}}(\partial_r\psi)\xi\dd r \dd z
 = \int_{\mathbb{R}^3}(\frac{\partial_r\psi}{r})\xi\dd\bbx
 =\int_{\mathbb{R}^3}v^z\xi\dd\bbx
$$
Thus to obtain $W>0$, it is enough to show 
$\int_{\mathbb{R}^3}v^z\xi\dd\bbx>0$.
From the representation of $v^z$ of \eqref{eq:uz-G},
\begin{equation*} 
\begin{split}
\int_{\mathbb{R}^3}v^z\xi\dd\bbx&=	 C_a \int_{\Pi}   \int_{\Pi} \left(  \frac{2(r-\bR)}{  r  (r\bR)^{\frac32-a}} \calF'(s) + \frac{(r\bR)^{a-\frac12}}{ r^2}\left[(a-\frac 12) \calF(s)-s \calF'(s) \right] \right) \xi(\bR,\bz) \bR\, \xi(r,z) r \ud \bR \ud \bz\ud r \ud z=I+II.\\
\end{split}
\end{equation*} 
For I, using symmetrization, we get
\begin{equation*} 
\begin{split}
I&=  C_a \int_{\Pi}   \int_{\Pi} \left(  \frac{2(r-\bR)}{  r  (r\bR)^{\frac32-a}} \calF'(s)   \right) \xi(\bR,\bz) \bR\, \xi(r,z) r \ud \bR \ud \bz\ud r \ud z\\&=
 C_a \int_{\Pi}   \int_{\Pi} {\left(  \frac{-(r-\bR)^2}{  r\bR  (r\bR)^{\frac32-a}} \underbrace{\calF'(s)}_{<0}   \right)} \underbrace{\xi(\bR,\bz) \bR\, \xi(r,z) r}_{\geq 0} \ud \bR \ud \bz\ud r \ud z>0.
\end{split}
\end{equation*} 
For II, we have
\begin{equation*} 
\begin{split}
II= 	 C_a \int_{\Pi}   \int_{\Pi} \left(    \frac{(r\bR)^{a-(1/2)}}{ r^2} {\left[\underbrace{(a-(1/2)) \calF(s)}_{>0}-\underbrace{s \calF'(s)}_{<0} \right]} \right) \underbrace{\xi(\bR,\bz) \bR\, \xi(r,z) r}_{\geq 0} \ud \bR \ud \bz\ud r \ud z>0.
\end{split}
\end{equation*} We obtained $W>0$.\\

To show boundedness of $\xi$, we recall that $W>0$ and 
$\xi=\Psi_+=(\psi-\frac 1 2Wr^2-\gamma)_+$. Thanks to the estimate
$\psi\lesssim( r^{1+a}+r^{2a})$ from \eqref{est_psi} of Lemma \ref{lem_est_stream}, we conclude that there exists some $R>1$ (depending on $W>0$ and $\mu>0$) such that
$$\mbox{spt }\xi\subset \{x\in\mathbb{R}^3\,|\, r\leq R\}.$$ It also implies $\|\psi\|_{L^\infty(
\{x\in\mathbb{R}^3\,|\, r\leq R\}
)}\lesssim R^{1+a}<\infty.$ Thus $\xi\in L^\infty(\mathbb{R}^3)$ from the relation $\xi =(\psi-\frac 1 2Wr^2-\gamma)_+$.
We note that $\psi$ (so $\Psi$) is continuous by Lemma \ref{lem_est_psi_con}, so is $\xi=\Psi_+$ (after possibly being redefined on a set of measure zero).\\

To show $\gamma\geq 0$, let's assume $\gamma <0$. By taking a  sequence $\{(r_n,z_n)\}$ such that 
$$\xi(r_n,z_n)\to 0, \quad r_n\to 0,
\quad \mbox{and} \quad z_n\to\infty\quad\quad\mbox{as $n$ goes to infinity}.$$ 
Since $|\psi|\lesssim (r^{1+a}+r^{2a})$ from Lemma \ref{est_psi}, we get
$$
|\psi-(1/2)Wr^2|\leq |\psi|+|(1/2)Wr^2|\to 0 \quad\quad \mbox{as $n$ goes to infinity}.
$$ Thus 
\begin{align*}
\lim_{n\to\infty}\left(\psi(r_n,z_n)-\frac 1 2 W r_n^2-\gamma\right)=-\gamma>0,
\end{align*} which contradicts to $\xi(r_n,z_n)\to 0$ with $\xi=\Psi_+$. Hence, we get $\gamma>0$.\\

Now we show uniqueness of the pair $(W,\gamma)\in \mathbb{R}^2$.  Suppose $\xi=(\psi-\frac 1 2 W r^2 -\gamma)_+=(\psi-\frac 1 2 W' r^2 -\gamma')_+$ for some $W,W',\gamma,\gamma'$. Thus, on the   set $\{\xi>0\}$ of positive measure, we have 
$$\psi-\frac 1 2 W r^2 -\gamma=\psi-\frac 1 2 W' r^2 -\gamma'$$ which gives 
$$\frac 1 2 (W'-W) r^2 =\gamma  -\gamma'$$ on the set. Thus we have the uniqueness: $W=W'$ and $\gamma=\gamma'$.\\

% \Red{remains to show     $\xi$ is   Lipshitz continuous, and compactly supported in $\mathbb{R}^3$}\\
  To show Lipschitz continuity of $\xi$, it's enough to show $\psi_R:=\psi|_{r<R},$ which  is the 
restriction of $\psi$ on the set $\{x\in\mathbb{R}^3\,|\, r\leq R\}
$, 
is Lipschitz on
$\{x\in\mathbb{R}^3\,|\, r\leq R\}
$. 
We recall $\Phi=(\phi_1,\phi_2,\phi_3)\in C^{1,\alpha}$ for some $\alpha>0$ from the regularity \eqref{Phi_reg} in  Lemma \ref{lem_est_psi_con}.  Since 
\begin{equation*}\label{psi_from_phi}
 \psi(r,z)=r  \phi_2(r e_{x_1}+z e_{x_3}),\quad (r,z)\in \Pi,
\end{equation*}  we get $\psi_R\in  C^{1,\alpha}(\{x\in\mathbb{R}^3\,|\, r\leq R\})$, which implies $\xi$ is Lipschitz continuous on $\mathbb{R}^3$. \\

%\psi=r\phi_2\in W^{1,\infty}$\\
In order to show $\xi$ is compactly supported,
  we recall the decay $$ \frac{\Phi(\bbx)}{r} \to 0\quad\mbox{as}\quad |\bbx|\to \infty$$  from \eqref{Phi_decay} in Lemma \ref{lem_est_psi_con}. On the other hand,
for any  $x \in \mbox{spt}(\xi)$, we know  $\psi(x) \geq \frac 1 2 W r^2$
i.e. $$|\Phi(\bbx)/r|=\psi/r^2\geq \frac 1 2 W>0.$$ Hence, $\xi$ is compactly supported in $\mathbb{R}^3$. This finishes the proof. 
\end{proof}
\begin{comment}
\begin{remark}
 %\begin{enumerate}
% \item
 $$0<W\mu+\gamma\int \xi\dd\bbx$$ by using $I_\mu>0$. Indeed,
 \begin{equation*}\begin{split}
 0<I_\mu&=\int \psi \xi \dd\bbx-\int \xi^2\dd\bbx
 =\int \psi \xi \dd\bbx-\int \xi(\psi-\frac 1 2 W r^2-\gamma )_+\dd\bbx
 \\& =\int \psi \xi \dd\bbx-\int \xi(\psi-\frac 1 2 W r^2-\gamma )\dd\bbx=\int \xi(\frac 1 2 W r^2+\gamma )\dd\bbx
 =W\mu+\gamma\int \xi\dd\bbx. 
\end{split} \end{equation*}
%\item  \end{enumerate}

\end{remark}
\end{comment}
\vspace{15pt} 

\subsection*{Acknowledgments}{D.~Chae was supported partially  by NRF grant 2021R1A2C1003234.
K.~Choi was supported    by NRF grants 2018R1D1A1B07043065 and 2022R1A4A1032094.
 I.-J.~Jeong was supported by the NRF grant 2022R1C1C1011051.}

% ----------------------------------------------------------------
\bibliographystyle{amsplain}

% ----------------------------------------------------------------

\end{document}